\documentclass[12pt,oneside]{amsart}

\usepackage[a4paper, margin=3cm]{geometry}
\usepackage{amsmath}
\usepackage{amsthm}
\usepackage{amsfonts}
\usepackage{amssymb}
\usepackage{mathabx}
\usepackage{mathtools}
\usepackage[normalem]{ulem}

\usepackage{graphicx}
\usepackage{caption}
\usepackage{subcaption}
\usepackage[export]{adjustbox}

\usepackage[nodayofweek]{datetime}
\usepackage{hyperref}
\usepackage{hyphenat}

\usepackage{color}

\definecolor{Green}{RGB}{30, 150, 30}

\newcommand{\sminus}{\mathbin{\setminus}}

\newtheorem{theorem}{Theorem}[section]
\newtheorem{lemma}[theorem]{Lemma}
\newtheorem{proposition}[theorem]{Proposition}
\newtheorem{corollary}[theorem]{Corollary}

\theoremstyle{definition}
\newtheorem{definition}[theorem]{Definition}
\newtheorem{claim}{Claim}
\newtheorem{example}[theorem]{Example}

\newcommand{\sep}{\operatorname{Sep}}

\newcommand{\DW}{\operatorname{DW}(S)} 
\newcommand{\FO}{\mathcal{F}_0(\Sigma)} 
\newcommand{\F}{\mathcal{F}(\Sigma)} 
\newcommand{\C}{\mathcal{C}}
\newcommand{\ksep}{\mathcal{K}}

\newcommand{\mcg}{\operatorname{MCG}}

\newcommand{\X}{\mathfrak{X}}

\newcommand{\G}{\mathcal{G}}

\newcommand{\diam}{\text{diam}}

\newcommand{\mc}[1]{\mathcal{#1}}
\newcommand{\mf}[1]{\mathfrak{#1}}
\newcommand*\Wit{\operatorname{Wit}}
\newcommand{\hierarchical}{hierarchical}
\newcommand{\Push}{\operatorname{Push}}

\newcommand{\tsh}[1]{\left\{\kern-.7ex\left\{#1\right\}\kern-.7ex\right\}}
\newcommand{\Tsh}[2]{\tsh{#2}_{#1}}
\newcommand{\ignore}[2]{\Tsh{#2}{#1}}

\makeatletter
\def\subsection{\@startsection{subsection}{2}
  \z@{.5\linespacing\@plus.7\linespacing}{.3\linespacing}
  {\normalfont\scshape}}
\makeatother

\title{The (non)-relative hyperbolicity of the separating curve graph}

\author{Jacob Russell}
\address{CUNY Graduate Center, 365 Fifth Avenue, New York, NY 10016, USA}
\email{jrussellmadonia@gradcenter.cuny.edu}

\author{Kate M. Vokes}
\address{IHES, 35 route de Chartres, 91440 Bures-sur-Yvette, France}
\email{vokes@ihes.fr}

\thanks{Date: 2 October 2019}

\begin{document}

\begin{abstract}
We prove that  the separating curve graph of a connected, compact, orientable surface with genus at least 3 and a single boundary component is not relatively hyperbolic. This completes the classification of when the separating curve graph is hyperbolic and relatively hyperbolic initiated by previous {works} of the authors. 
\end{abstract}

\maketitle

\section{Introduction}

For a connected, compact, orientable surface $S=S_{g,b}$ with genus $g$ and $b$ boundary components, the separating curve graph, $\sep(S)$, is the metric graph whose vertices are all isotopy classes of separating curves, with edges of length~1 corresponding to disjointness.
The separating curve graph  arises naturally in the study of the Johnson kernel of the mapping class group~\cite{brendlem,Kida_Johnson_Kernel}, the coarse geometry of the Weil--Petersson metric on Teichm\"uller space~\cite{Brock_Masur_WP_Low_complexity,bowditch_wprigidity,sultanthesis}, and the algebraic topology of the moduli space of a surface~\cite{looijenga}.

When  $S$ has genus $0$, all curves on the surface are separating and the separating curve graph is the same as the famous curve graph, $\C(S)$, which is defined similarly using all curves on~$S$ instead of just the separating curves. In this case, $\sep(S) = \C(S)$ is a Gromov hyperbolic metric space by a celebrated theorem of Masur and Minsky~\cite{mm1}. While the curve graph is hyperbolic for all surfaces, it has long been understood that the  separating curve graph is not in general hyperbolic (see~\cite[Exercise~2.42]{schleimer_notes}).

The second author illuminated the coarse geometry of the separating curve graph by showing it is a hierarchically hyperbolic space in all cases where $\sep(S)$ is non-empty \cite{vokessep}.
Hierarchical hyperbolicity is a notion of non\hyp{}positive curvature introduced by Behrstock, Hagen, and Sisto to generalize Masur and Minsky's subsurface projection machinery for the mapping class group~\cite{hhs1,hhs2,mm2}.
Every hierarchically hyperbolic space is equipped with a family of projection maps to Gromov hyperbolic spaces. These projection maps satisfy a list of axioms that allows for the geometry of the space to be recovered from the images of these projections.  The simplest examples of hierarchically hyperbolic spaces are hyperbolic spaces, where the projection map can be taken to be the identity map from the space to itself. The projection maps for the mapping class group are the subsurface projections, to curve graphs of all subsurfaces, defined by Masur and Minsky in~\cite{mm2}.  
The hierarchically hyperbolic structure  on $\sep(S)$  constructed by the second author  also uses Masur and Minsky's subsurface projections, this time to curve graphs of only certain subsurfaces.

As a consequence of this hierarchically hyperbolic structure, the second author deduced that the separating curve graph is hyperbolic when the surface has at least three boundary components \cite{vokessep}.
For surfaces with genus at least~3 and zero or two boundary components, the first author used the hierarchically hyperbolic structure to show the separating curve graph is relatively hyperbolic \cite{russell}.
Like a relatively hyperbolic group, a relatively hyperbolic space is hyperbolic outside of a collection of isolated peripheral subsets (see~\cite{SistoMetRel} for several equivalent definitions).
While hierarchical hyperbolicity is a generalization of hyperbolic and relatively hyperbolic spaces, these results demonstrate that it can be advantageous to first prove a space is hierarchically hyperbolic and then use the tools provided by this to show (relative) hyperbolicity.

Our new addition to this theory is the resolution of the final case, proving the separating curve graph is not relatively hyperbolic for surfaces with one boundary component.

\begin{theorem}\label{intro_theorem:sep(S)_is_thick}
	The separating curve graph of a surface with one boundary component and genus at least 3 is not hyperbolic or relatively hyperbolic.
\end{theorem}

We establish Theorem~\ref{intro_theorem:sep(S)_is_thick} by showing the separating curve graph is \emph{thick} when the surface has exactly one boundary component. Thickness was originally introduced by Behrstock, Dru\c{t}u and Mosher as an obstruction to a space being relatively hyperbolic, but has emerged as a powerful quasi-isometry invariant of metric spaces \cite{Behrstock_Drutu_Mosher_Thickness,Behrstock_Drutu_Short_Conjugators,BHS_RACG}. Thickness is defined inductively, with spaces that are products of two infinite diameter metric spaces being thick of order~0. A metric space  $X$ is thick of order at most~$n\geq 1$ if any two points in $X$ can be connected by a ``thick chain" of subsets that are thick of order at most~$n-1$. That is, for any $x,y\in X$, there exists a sequence of thick of order at most~$n-1$ subsets $P_1,\dots,P_k$, so that $x$ is in $P_1$, $y$ is in $P_k$, and the coarse intersection between $P_i$ and $P_{i+1}$ has infinite diameter for each $1\leq i \leq k-1$. If a space is thick of any order, then it cannot be relatively hyperbolic \cite[Corollary 7.9]{Behrstock_Drutu_Mosher_Thickness}.

We show that $\sep(S_{g,1})$ is thick of order at most~2. The thick of order~0 subsets are product regions that arise naturally from the hierarchically hyperbolic structure.  These product regions are in one-to-one correspondence with pairs of disjoint \emph{witnesses}, connected subsurfaces that intersect every separating curve on~$S_{g,1}$. The thick chains of these product regions can be understood by examining sequences of sequentially disjoint witnesses. Not every two points in $\sep(S_{g,1})$ can be joined by a thick chain of these product regions, so we form thick of order~1 subsets by taking unions of product regions that can be thickly chained together. The main difficulty is then understanding how these thick of order~1 subsets chain together. We resolve this by using the map from $S_{g,1}$ to the closed surface $S_{g,0}$ given by capping off the boundary component.
The fibers of this map illuminate the chaining of the thick of order~1 subsets we have constructed in~$\sep(S_{g,1})$.  This method is an adaption of the argument in \cite[Proposition~3]{Brock_Masur_WP_Low_complexity}, which exploits the capping map to prove the thickness of the pants graph of~$S_{2,1}$.

By combining the present work with previous works of the authors, we obtain a complete characterization of the (relative) hyperbolicity of the separating curve graph in Theorem~\ref{intro_theorem:geometry of sep(S)} below. The hyperbolicity of $\sep(S)$ was previously understood for $S= S_{2,0}$~\cite{sultan_genus_two}, and in the genus~0 case, when $\sep(S) = \C(S)$ \cite{mm1}. Theorem~\ref{intro_theorem:geometry of sep(S)} covers all possible cases, since $S_{g,b}$ contains no separating curves when $2g+b <4$. As we discuss in Section~\ref{section: sep connectedness}, when $(g,b) \in \{(0,4), (1,2), (2,0), (2,1)\}$, $\sep(S_{g,b})$ is non-empty, but the edge relation of disjointness  produces a graph that is not connected. To achieve a connected graph in these exceptional cases, we modify the edge relation to put an edge between any two curves that intersect at most four times for $S_{1,2}$, $S_{2,0}$ and $S_{2,1}$, and at most twice for $S_{0,4}$. While the connectedness of $\sep(S)$ when $2g+b\ge 4$ is well known to experts, we could not find a proof in the literature covering all of the cases in Theorem~\ref{intro_theorem:geometry of sep(S)}. 
In Section \ref{section: sep connectedness}, we use a technique of Putman to provide a unified proof of connectedness  in all the cases where $\sep(S)$ is not equal to~$\C(S)$.

\begin{theorem}\label{intro_theorem:geometry of sep(S)}
	Let $S=S_{g,b}$ be so that $\sep(S)$ is non\hyp{}empty, that is,  $2g+b \geq 4$.
	\begin{enumerate}
		\item If $b \geq 3$ or $(g,b) \in \{(1,2), (2,0), (2,1)\}$, then $\sep(S)$ is hyperbolic \cite[Example~2.4]{vokessep}.
		\item  If $b=0$ and $g\geq 3$, then $\sep(S_{g,b})$ is relatively hyperbolic with peripherals quasi-isometric to $\C(S_{0,g+1})\times \C(S_{0,g+1})$ \cite[Theorem~6.8]{russell}.
		\item  If $b=2$ and $g\ge2$, then $\sep(S_{g,b})$ is relatively hyperbolic with peripherals quasi-isometric to $\C(S_{0,g+2})\times \C(S_{0,g+2})$ \cite[Theorem~6.8]{russell}.
		\item If $b=1$ and $g\geq 3$, then $\sep(S_{g,b})$ is thick of order at most 2. 
	\end{enumerate}
\end{theorem}

After laying out some preliminaries in Section~\ref{section: preliminaries}, we prove connectedness of $\sep(S)$ for all surfaces with genus at least~1 in Section~\ref{section: sep connectedness}. In Section~\ref{section: hierarchies}, we detail the hierarchy structure on the separating curve graph established by the second author in~\cite{vokessep}. This entails describing a general class of graphs of multicurves from which we find a quasi-isometric model for the separating curve graph and a concrete description of the product regions. Finally, we show the separating curve graph of a surface with one boundary component is thick in Section~\ref{section: thickness}.

\section{Preliminaries} \label{section: preliminaries}

Throughout, we will consider connected, orientable, finite type surfaces. As the separating curve graph does not distinguish between a boundary component and a puncture, we may assume that all surfaces are compact.  Thus, each surface will be homeomorphic to some $S_{g,b}$ with genus $g$ and $b$ components. The \emph{complexity} of $S_{g,b}$ is $\xi(S_{g,b})= 3g-3 +b$.

By a \emph{curve} on a surface $S$ we mean an isotopy class of essential, non\hyp{}peripheral simple closed curves on~$S$. A \emph{subsurface} of $S$ will be an isotopy class of compact subsurfaces.
We assume all subsurfaces are \emph{essential}, that is, all boundary components of the subsurface are either boundary components of $S$ or essential, non\hyp{}peripheral curves.
We say curves and/or subsurfaces are \emph{disjoint} if they have disjoint representatives. A \emph{multicurve} of $S$ is a collection of pairwise disjoint, pairwise non-isotopic curves on~$S$. If a multicurve $\mu$ and a subsurface $Y$ are not disjoint,  we say $\mu$ \emph{intersects}~$Y$.  Abusing notation, if $\mu$ is a multicurve on~$S$, $S \sminus \mu$ will denote  the complement of a regular open neighborhood of~$\mu$.
Similarly, for a subsurface $Y$ of $S$, $S \sminus Y$ will denote the closure of the complement of~$Y$. A multicurve $\mu$ is \emph{separating} if $S \sminus \mu$ is disconnected. The \emph{intersection number} of two multicurves $\mu$ and $\nu$ on $S$ is denoted $i(\mu,\nu)$ and is the minimal number of intersections between two representatives of $\mu$ and~$\nu$. Given a subsurface $Y$ of $S$, we denote by $\partial_S Y$ the multicurve on $S$ composed of boundary curves of $Y$ that are not also boundary curves of~$S$.

The \emph{curve graph}, $\C(S)$, of a surface $S$ has a vertex for every curve on $S$ and an edge joining two vertices if they are disjoint. 
We make the standard modification for $S_{1,0}$, $S_{1,1}$ and $S_{0,4}$ by putting an edge between two vertices which intersect minimally.
The \emph{separating curve graph}, $\sep(S)$, is defined similarly using only the curves that are separating.
We give the precise edge relation for the separating curve graph in Section~\ref{section: sep connectedness}.
All graphs will be considered as metric spaces by declaring each edge to have length~1.

For every connected subsurface $Y$ of $S$ with $\C(Y)$ non-empty, Masur and Minsky defined a \emph{subsurface projection} map $\pi_Y \colon \C(S) \to 2^{\C(Y)}$. We recall a few properties and direct the reader to \cite[Section~2.3]{mm2} for full details.
For a set of curves $A$ on~$S$, we define $\pi_Y(A)=\bigcup_{\alpha \in A}\pi_Y(\alpha)$. If $\mu$ is a multicurve on~$S$, then $\pi_Y(\mu)$ is empty if $\mu$ is disjoint from $Y$ and is a non-empty subset of diameter at most~3 if $\mu$ intersects~$Y$. If $\mu$ and $\nu$ are two multicurves on~$S$ that both intersect a subsurface~$Y$, then we define $d_Y(\mu,\nu)=\operatorname{diam}_{\C(Y)}(\pi_Y(\mu) \cup \pi_Y(\nu))$.

We define the \emph{mapping class group} of~$S$, $\mcg(S)$, to be the group of isotopy classes of orientation\hyp{}preserving self\hyp{}homeomorphisms of~$S$ (note, we do not require these to fix the boundary pointwise). The action of $\mcg(S)$ on the set of curves on~$S$ induces an action of $\mcg(S)$ on $\C(S)$ by isometries. The \emph{pseudo-Anosov} elements of $\mcg(S)$ are precisely those that act loxodromically on~$\C(S)$ \cite[Proposition~4.6]{mm1}.

For the surface $S_{g,1}$, there is a natural map $F \colon \C(S_{g,1}) \to \C(S_{g,0})$ and a homomorphism $\mcg(S_{g,1}) \to \mcg(S_{g,0})$, both induced by capping the boundary component of $S_{g,1}$ with a disk.  The kernel of this homomorphism is the \emph{point pushing subgroup} of $\mcg(S_{g,1})$ \cite{Birman_Exact_Sequence}. We will denote the point pushing subgroup by $\Push(S_{g,1})$.
If $\mu$ is a multicurve on $S_{g,1}$ and $\phi\in\Push(S_{g,1})$, then $F(\mu) = F(\phi(\mu))$.

\section{Connectedness of the separating curve graph} \label{section: sep connectedness}

For completeness, we give a proof of the connectedness of the separating curve graph of $S=S_{g,b}$, with the appropriate definition of edges.
There {exist} proofs in the literature for various cases, but we could not find one that covers every case where $\sep(S)$ is non\hyp{}empty.
If $g=0$, then the separating curve graph is the same as the curve graph, and we refer the reader to \cite[Lemma~2.1]{mm1} when $b\ge5$, and \cite[Section~3]{minskygeometric} for~$S_{0,4}$.
A result implying connectedness of the separating curve graph for surfaces of genus at least 3 was announced by Farb and Ivanov in~\cite{farbivanov}. Numerous proofs have been given for the case of closed surfaces with genus at least~3 \cite{mccarthyv,mspants,putman}, and stronger connectivity results are proved in~\cite{looijenga} for surfaces of genus at least~2 that are not $S_{2,0}$ or~$S_{2,1}$. A proof for the cases where the edge relation of disjointness gives a connected graph previously appeared in~\cite{vokesthesis}. The proof we present here is a unified proof for all of the cases where the separating curve graph is non\hyp{}empty and {is} not the curve graph.

\begin{definition}
Let $S = S_{g,b}$ and define $\sep_K(S)$ to be the graph whose vertices are all separating curves on $S$, with two vertices joined by an edge if they intersect at most $K$ times. If the set of $K$ such that $\sep_K(S)$ is connected is non\hyp{}empty, we define $K_0$ to be the minimal value in this set, and define $\sep(S)=\sep_{K_0}(S)$.
\end{definition}

\begin{theorem} \label{theorem: sep connected}
Let $S=S_{g,b}$, with $g \ge 1$ and $2g+b \geq 4$.
\begin{itemize}
    \item If $(g,b) \in \{(1,2),(2,0),(2,1)\}$, then $\sep(S) = \sep_4(S)$.
    \item Otherwise $\sep(S)=\sep_0(S)$. 
\end{itemize}
\end{theorem}

The proof employs the following trick of Putman to establish connectedness.

\begin{lemma}[{\cite[Lemma~2.1]{putman}}]\label{lem:putman_trick}
Let $\G$ be a simplicial graph and suppose the group $G$ acts on $\G$ by simplicial automorphisms. Fix a vertex $v_0 \in \G$ and a symmetric generating set $X$ for $G$. Suppose that:
\begin{itemize}
    \item for all vertices $v \in \G$, the orbit $G \cdot v_0$ intersects the connected component of $\G$ containing $v$;
    \item for all $g \in X$, $v_0$ is connected to $g\cdot v_0$ in $\G$.
\end{itemize}
Then, the graph $\G$ is connected.
\end{lemma}

In our case, $G$ will be the mapping class group, $\mcg(S)$, and the generating set $X$ will be the left and right Dehn twists around each of the curves shown in Figure~\ref{figure: generators} plus half twists exchanging any two boundary components.
This is an extension of the Humphries generating set to the case of non-closed surfaces; see, for example, \cite[Section~4.4.4]{primer}.

\begin{figure}[!h]
\centering
\includegraphics[width=.5\textwidth]{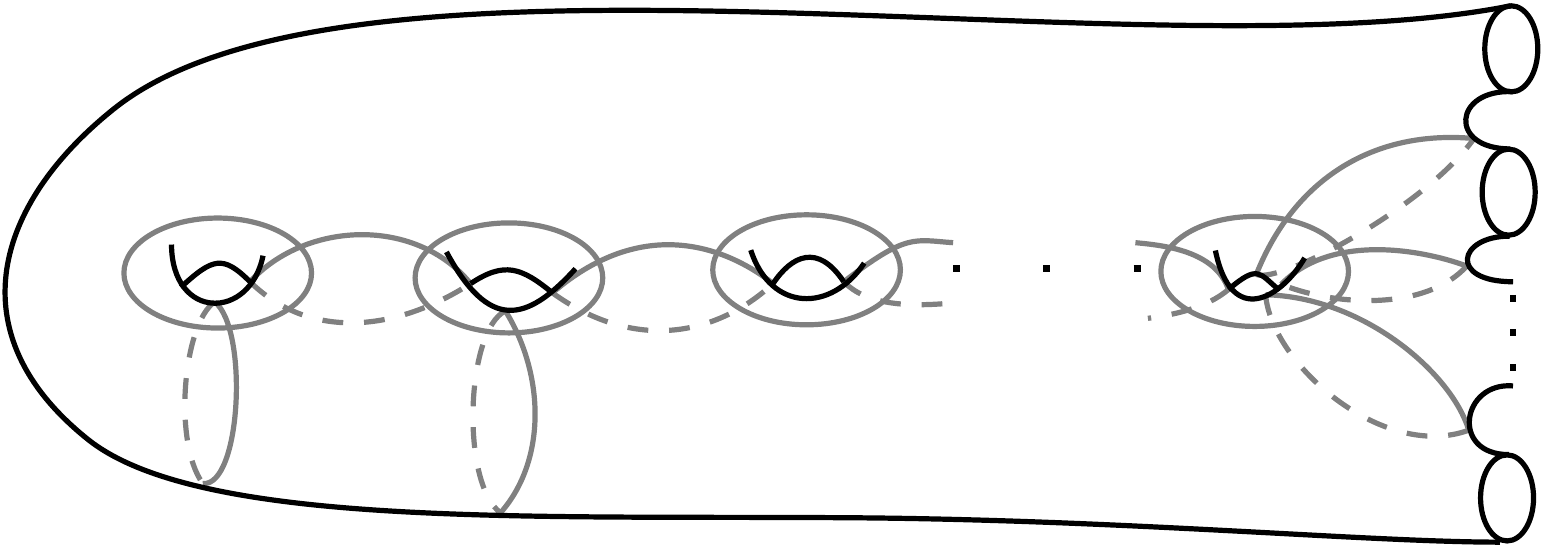}
\caption{Generating Dehn twists for $\mcg(S)$.}
\label{figure: generators}
\end{figure}

\begin{proof}[Proof of Theorem~\ref{theorem: sep connected}]
Let $X$ be the generating set for $\mcg(S)$ described above. 
Fix a base vertex $\alpha$ of $\sep(S)$ as shown in Figure~\ref{figure: generatorswithcurve} so that $\alpha$ cuts off a \emph{handle}, a subsurface homeomorphic to~$S_{1,1}$.

\begin{figure}[!h]
\centering

\begin{subfigure}[b]{0.6\textwidth}
\def\svgscale{.5}
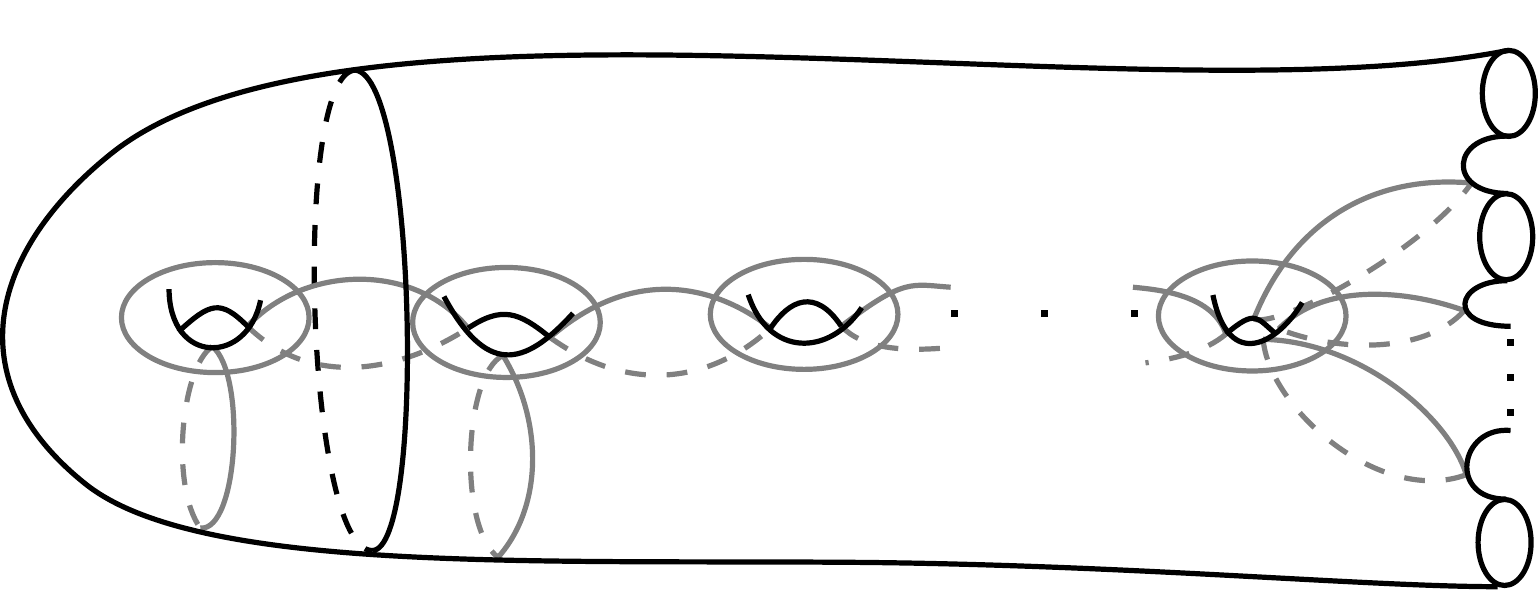
\subcaption{Genus at least 2.}
\label{figure: generatorswithcurve2}
\end{subfigure}
\quad
\begin{subfigure}[b]{0.3\textwidth}
\def\svgscale{.5}
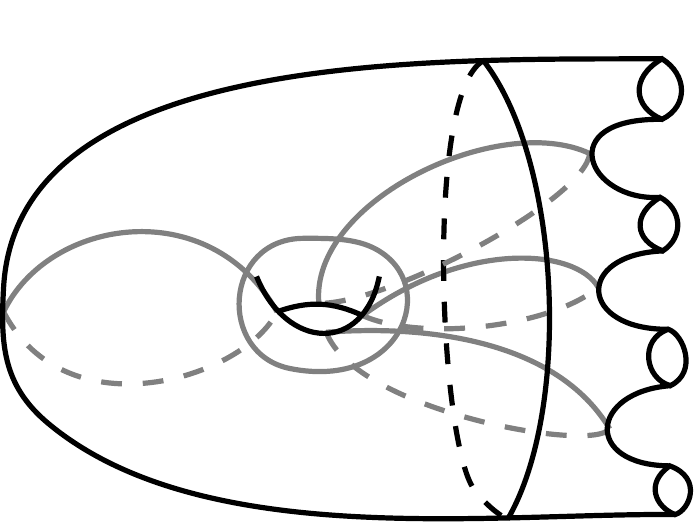
\subcaption{Genus 1.}
\label{figure: generatorswithcurve1}
\end{subfigure}

\caption{A choice of base vertex $\alpha$ for $\sep(S)$.}\label{figure: generatorswithcurve}
\end{figure}

The mapping class group orbit of a separating curve is determined by the topological type of its complement.
For every separating curve~$\alpha'$, $S \sminus \alpha'$ has a component that contains a handle. Hence, $\alpha'$ is either in the $\mcg(S)$\hyp{}orbit of $\alpha$ or connected to the $\mcg(S)$\hyp{}orbit of $\alpha$ by an edge.

Now, let $\phi$ be an element of our generating set $X$. If $\phi$ exchanges  two boundary components of $S$, then $\phi(\alpha) = \alpha$. If $\phi$ is a left or right Dehn twist about one of the generating curves shown in Figure~\ref{figure: generators}, then $i(\alpha, \phi(\alpha))\le 4$ as any of these curves intersects $\alpha$ at most twice. Thus, $\sep_K(S)$ will be connected when~$K \geq 4$ by Lemma~\ref{lem:putman_trick}.

We now prove that the minimal choice of $K$ is as asserted in the theorem.
If $S$ is not $S_{2,0}$, $S_{2,1}$ or $S_{1,2}$, then we want to show that, for any of our generators $\phi$, there is a sequence of separating curves joining $\alpha$ and $\phi(\alpha)$ so that consecutive curves are disjoint.
If $\phi$ is an exchange of boundary components or a Dehn twist about a curve that is disjoint from $\alpha$, then $\phi(\alpha)=\alpha$ and we are done.
Otherwise, $\phi$ is a left or right Dehn twist about one of the generating curves $\gamma$ that intersects $\alpha$ twice (see Figure~\ref{figure: generatorswithcurve}).
We claim that when $S$ is not $S_{2,0}$, $S_{2,1}$ or $S_{1,2}$, there is a separating curve disjoint from both $\alpha$ and $\phi(\alpha)$.
The curves $\alpha$ and $\gamma$ fill a $4$\hyp{}holed sphere. Since $\alpha$ cuts off an handle, two of the boundary components of this  $4$\hyp{}holed sphere are identified to form a $S_{1,2}$~subsurface $Y$ of~$S$ with boundary components $\delta_1$ and $\delta_2$ (see Figure~\ref{figure:sep:disjoint from twist}).
This subsurface $Y$ contains both $\alpha$ and $\phi(\alpha)$. We shall argue that $S \sminus Y$ or $\partial_S Y$ must contain a separating curve of~$S$.

Suppose $g \geq 2$. Since $S$ is not $S_{2,0}$ or $S_{2,1}$, then $S \sminus Y$ is a connected subsurface with either genus at least~1 or at least four boundary components.  In either case, there exists a curve $\alpha' \subset S \sminus Y$  that separates the curves $\delta_1$ and $\delta_2$ from the genus or the other boundary curves of $S\sminus Y$ (see Figure~\ref{figure:sep:disjoint from twist2}).  If $g = 1$, then $b \geq 3$  since $S \neq S_{1,2}$ and $2g+b \geq 4$.
In this case, $S \sminus Y$ must have a component containing at least two boundary components of~$S$. Thus, there is a curve $\alpha'$ in $S \sminus Y$ or $\partial_S Y$ that cobounds a pair of pants with two boundary components of~$S$  (see Figure~\ref{figure:sep:disjoint from twist1}).
In all cases, $\alpha'$ is a separating curve on $S$ that is disjoint from $Y$ and hence disjoint from both $\alpha$ and~$\phi(\alpha)$. This deals with all of the cases where $\sep(S)=\sep_0(S)$.

\begin{figure}[h!]
\centering

\begin{subfigure}[b]{0.6\textwidth}
\def\svgscale{.5}
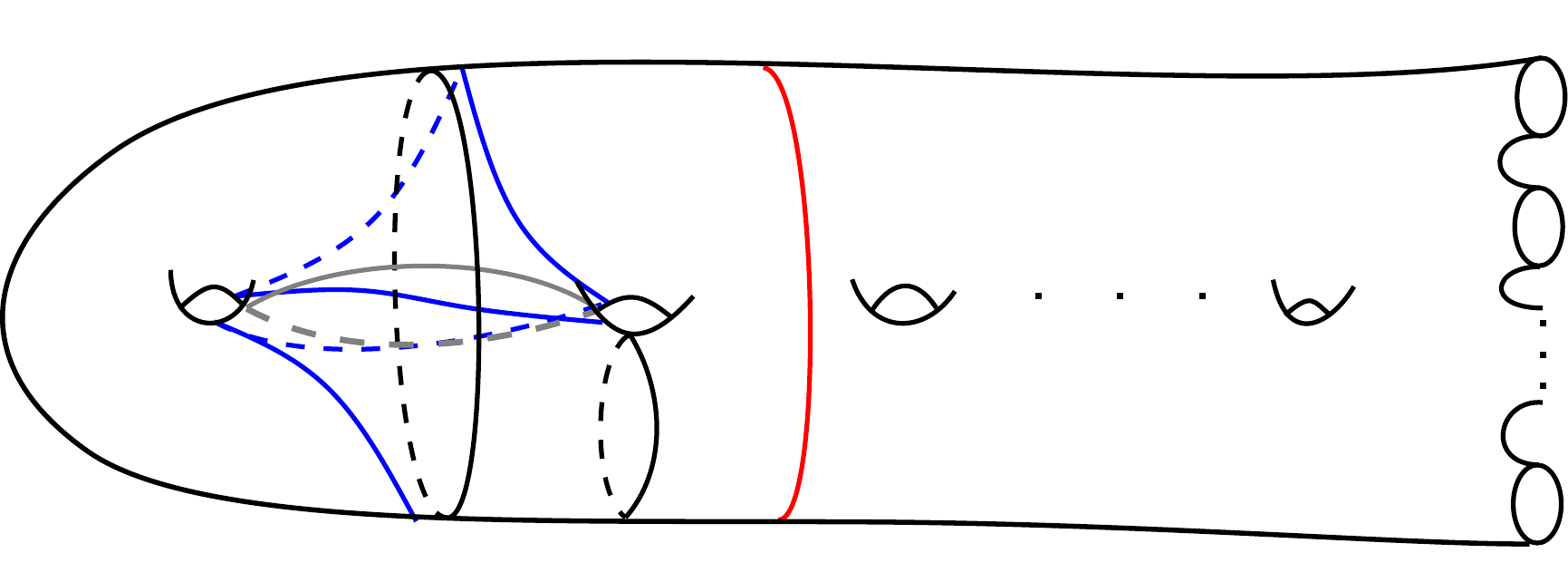
\subcaption{Genus at least 2.}
\label{figure:sep:disjoint from twist2}
\end{subfigure}
\quad
\begin{subfigure}[b]{0.3\textwidth}
\def\svgscale{.5}
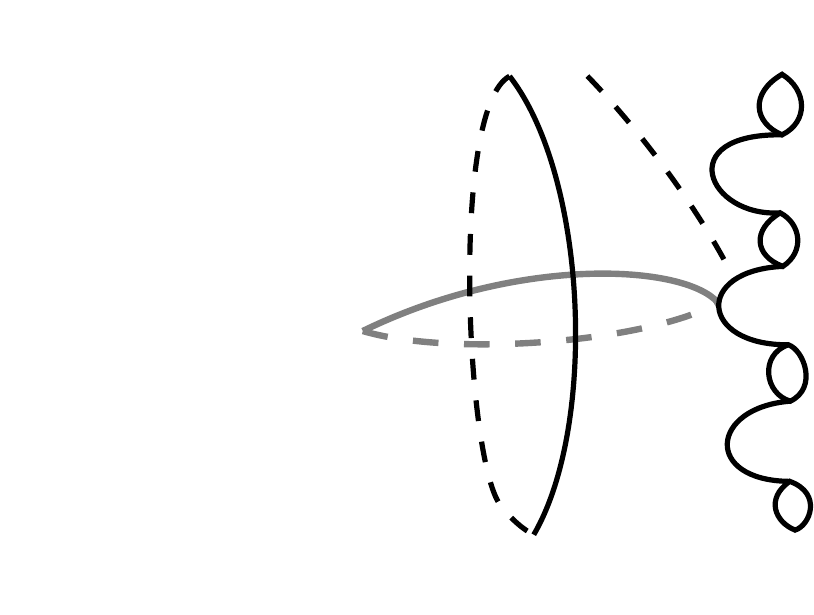
\subcaption{Genus 1.}
\label{figure:sep:disjoint from twist1}
\end{subfigure}

\caption{The union of $\alpha$ and its image under a generating Dehn twist $\phi$ is always disjoint from another separating curve $\alpha'$ when $S$ is not $S_{2,0}$, $S_{2,1}$, or $S_{1,2}$.}
\label{figure:sep:disjoint from twist}
\end{figure}

Now, suppose $S$ is $S_{1,2}$, $S_{2,0}$ or~$S_{2,1}$.
In the first two cases, $\sep_0(S)$ contains no edges, since there are no pairs of disjoint separating curves on~$S$.
For $S=S_{2,1}$, the graph $\sep_0(S)$ does contain edges, but is known to not be connected \cite[Lemma~3.4]{sultanthesis}.
Thus, to show that $\sep(S) = \sep_4(S)$ in these cases, it is sufficient to show that no two separating curves on $S$ have intersection number exactly~2.
Let $\alpha$ be any separating curve on~$S$. In the cases we are considering, $S \sminus \alpha$ must have a component that is a copy of~$S_{1,1}$.
Denote this $S_{1,1}$ component by~$Z$. If $\beta$ is a curve on $S$ with $i(\alpha,\beta) = 2$, then $\beta \cap Z$ is a single arc and there exists a curve $\gamma \subseteq Z$ which intersects the arc $\beta \cap Z$ exactly once. This implies $\beta$ cannot be separating as $\beta$ and $\gamma$ intersect exactly once on~$S$. Thus,  no two separating curves on $S_{1,2}$, $S_{2,0}$ or $S_{2,1}$ have intersection number exactly~2. 
\end{proof}

\section{Hierarchical Graphs of Multicurves}\label{section: hierarchies}

We now describe a broad class of graphs associated to surfaces introduced in \cite{vokessep}.  While we are primarily interested in the separating curve graph, this greater level of generality is helpful for understanding the thick subsets of $\sep(S)$. A \emph{graph of multicurves} on a surface $S = S_{g,b}$ is a non\hyp{}empty graph whose vertices are multicurves on~$S$.
If $\G(S)$ is a graph of multicurves on~$S$, then we say a {connected} subsurface $W \subseteq S$ is a \emph{witness} for $\G(S)$ if $W$ is not homeomorphic to $S_{0,3}$ and every vertex of $\G(S)$ intersects~$W$ (in other words, $W$ is a witness if every vertex of $\G(S)$ has non\hyp{}trivial subsurface projection to~$W$ \cite[Section~2]{mm2}).
If $\xi(S)\ge1$, the entire surface $S$ will always be a witness for~$\G(S)$.
We denote the set of  witnesses for $\G(S)$ by $\Wit\bigl(\mc{G}(S)\bigr)$. 

The next proposition concretely describes the witnesses for the separating curve graph; examples and non-examples are given in Figure~\ref{figure: holes-nonholes}.

\begin{proposition}\label{proposition: sep holes}
Let $S = S_{g,b}$ with $2g+b \geq 4$. A connected subsurface $Y \subseteq S$ is a witness for $\sep(S)$ if and only if $Y$ is not a copy of~$S_{0,3}$ and every component of $S \sminus Y$ is planar and contains at most one boundary component of~$S$. In particular, no witness for $\sep(S)$ is an annulus.
\end{proposition}

\begin{proof}
Let $Y$ be a witness for $\sep(S)$.
Then every component of $\partial_S Y$ is non\hyp{}separating in $S$ and no component of $S \sminus Y$ contains a separating curve of~$S$.
In particular, no component of $S \sminus Y$ can have positive genus or contain more than one component of~$\partial S$.

Conversely, let $Y$ be a connected subsurface where every component of $S \sminus Y$ is planar and contains at most one component of~$\partial S$.
If $\alpha$ is a separating curve disjoint from~$Y$, then one component $Z$ of $S \sminus \alpha$ is a subsurface of a component of~$S \sminus Y$.
Since $\alpha$ is separating, $Z$ either has genus or contains at least two boundary components of~$S$. 
This contradicts the assumption that each component of $S \sminus Y$ is planar and contains at most one boundary component of $S$.
Hence no such separating curve exists and $Y \in \Wit\bigl(\sep(S)\bigr)$.

Now, suppose  $Y \in \Wit\bigl( \sep(S) \bigr)$ is an annulus.
From above, every component of $S \sminus Y$ is planar and contains at most one component of $\partial S$.
The only possibilities are either that $S=Y$ is an annulus, or $S = S_{1,0}$ or~$S_{1,1}$ and $S \sminus Y$ is an annulus or pair of pants meeting $Y$ along two curves.
These cases are all excluded by our hypotheses on~$S$.
\end{proof}

\begin{figure}[h!]
\centering
\begin{subfigure}[b]{0.48\textwidth}
\centering
\includegraphics[width=.5\textwidth]{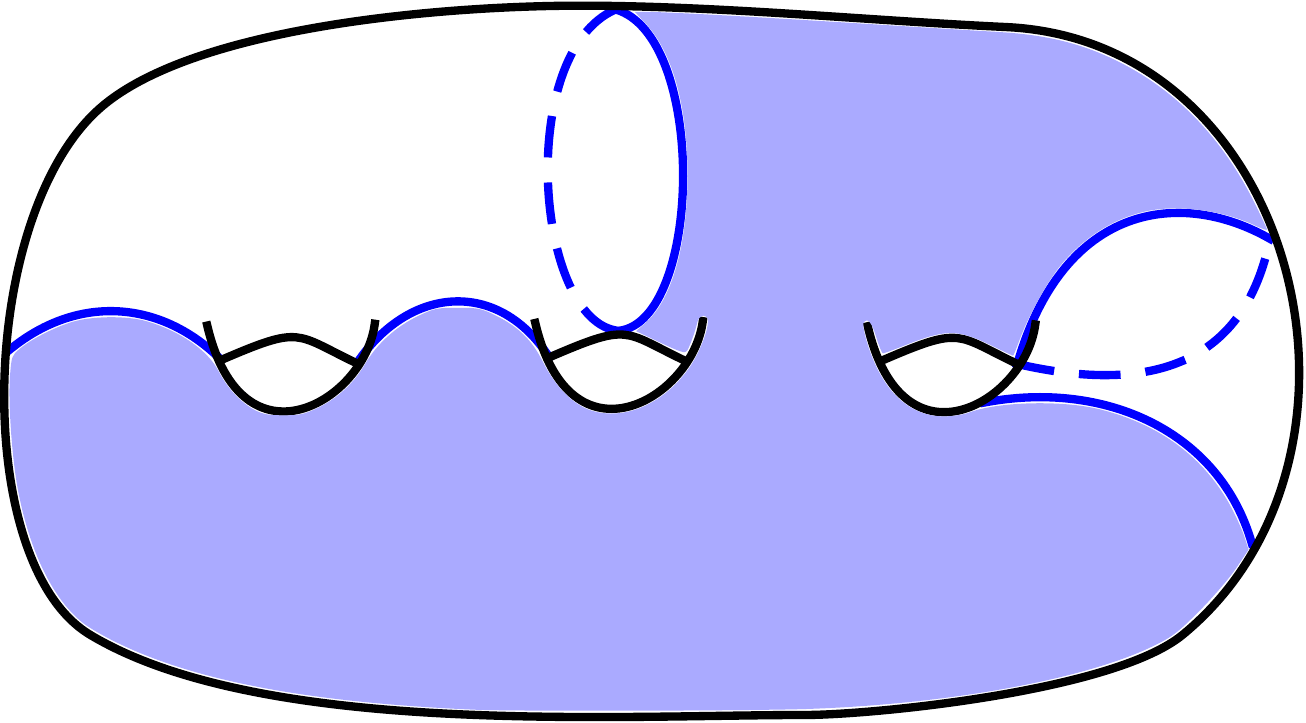}
\includegraphics[width=.35\textwidth]{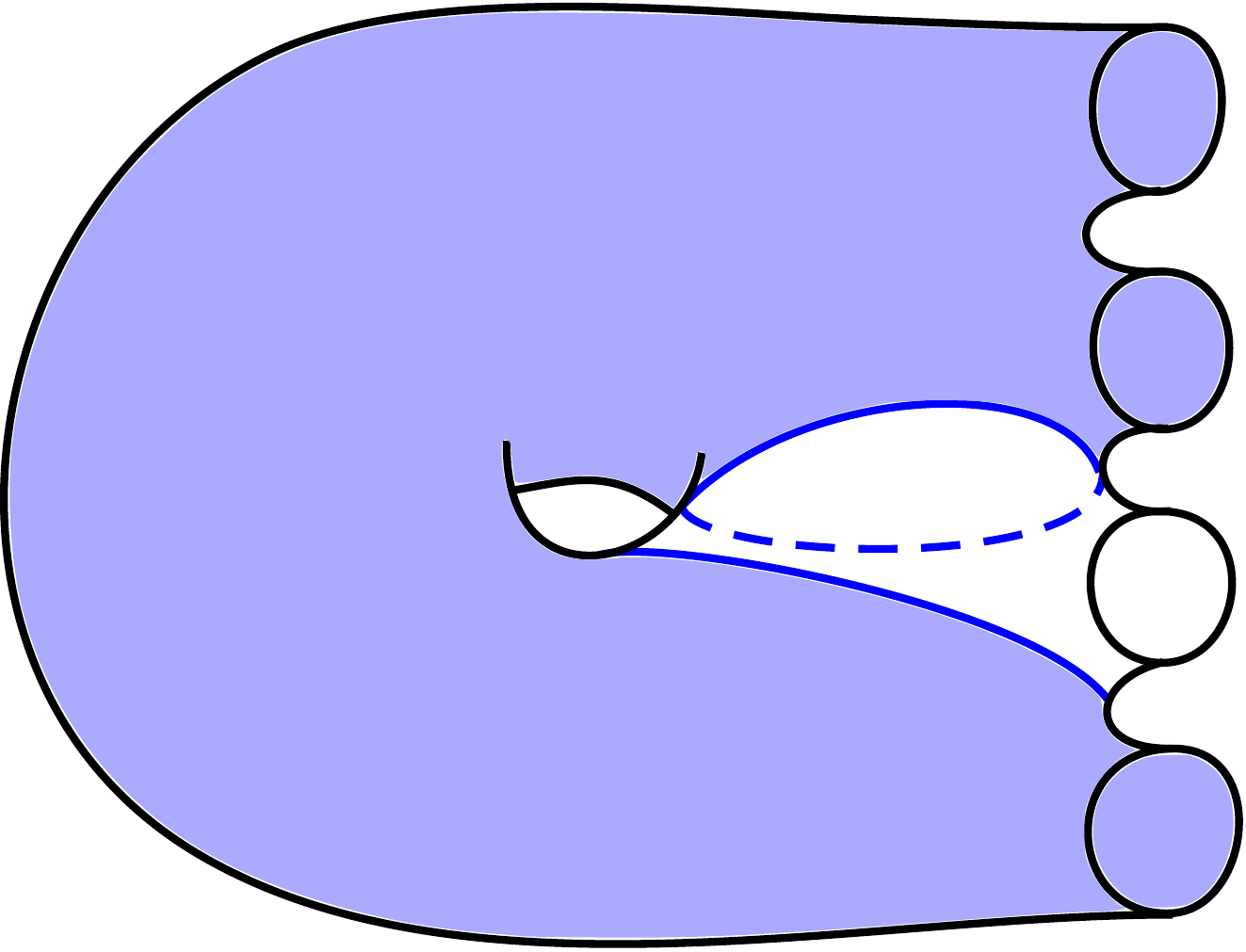}
\subcaption{Witnesses for $\sep(S)$.}
\label{figure: holes}
\end{subfigure}
\begin{subfigure}[b]{0.48\textwidth}
\centering
\includegraphics[width=.5\textwidth]{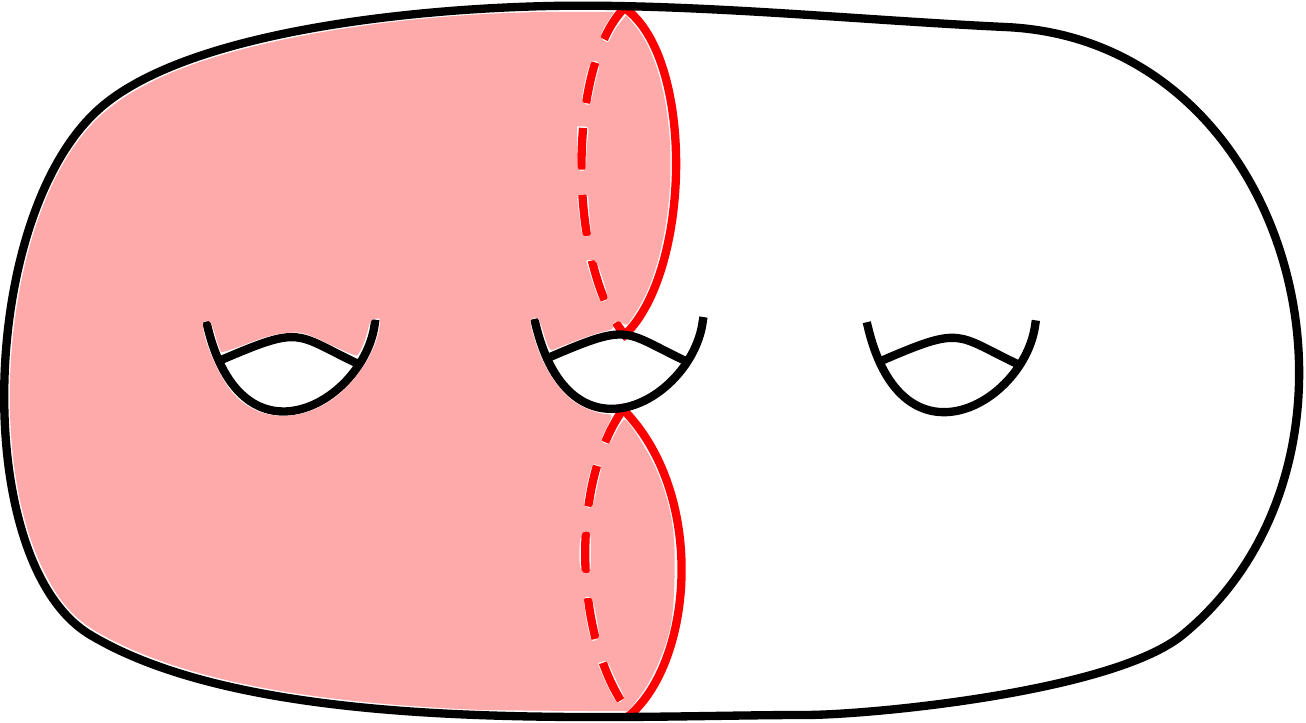}
\includegraphics[width=.35\textwidth]{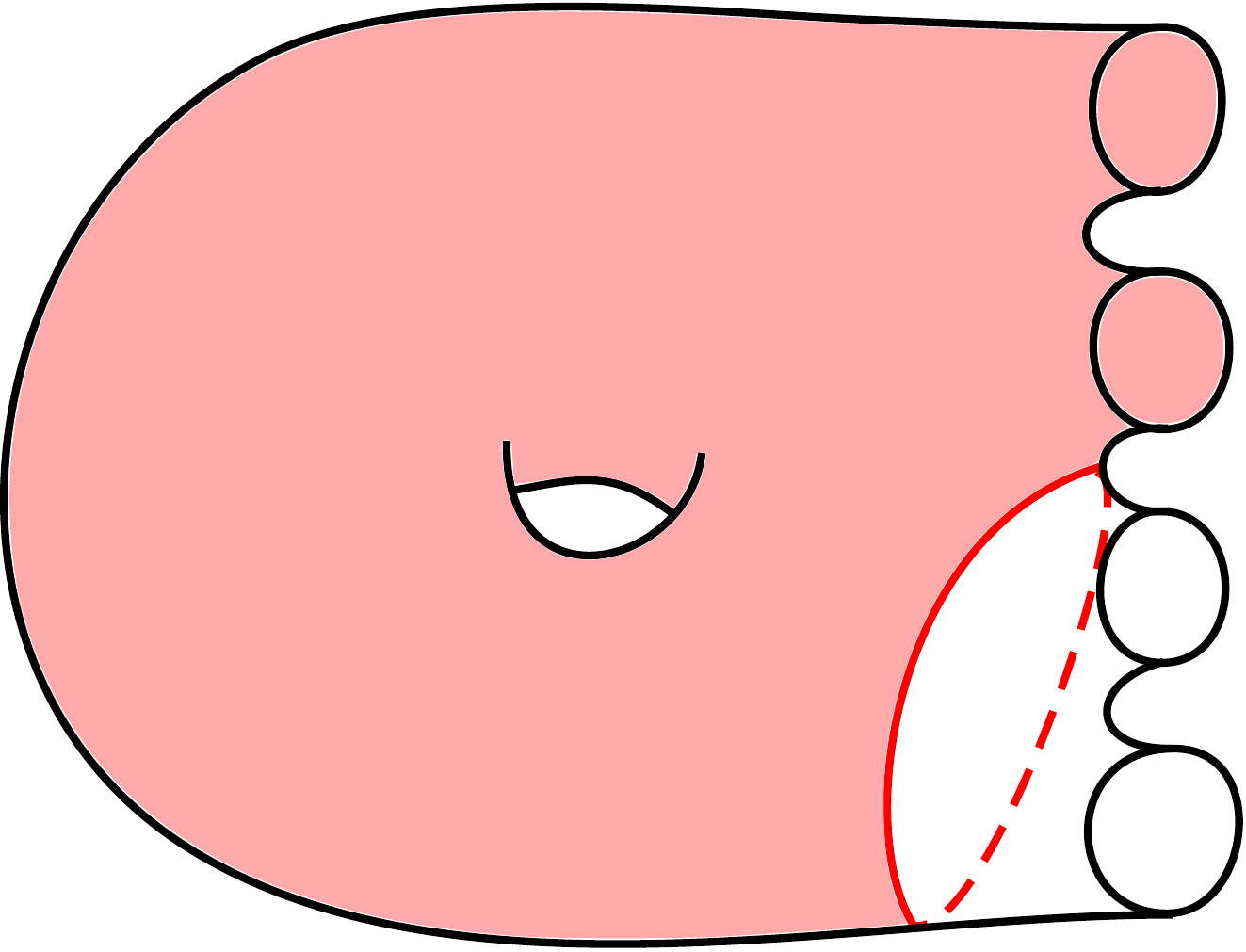}
\subcaption{Non-witnesses for $\sep(S)$.}
\label{figure: nonholes}
\end{subfigure}
\caption{}
\label{figure: holes-nonholes}
\end{figure}

In \cite{vokessep}, the second author provides a construction to produce a graph of multicurves on $S$ whose witnesses contain a specified collection of subsurfaces.

\begin{definition} \label{definition: ksep}
Let $\mf{S}$ be a collection of connected subsurfaces of $S = S_{g,b}$ with $\xi(W) \geq 1$ for all $W \in \mf{S}$.
If $\mf{S} = \emptyset$, define $\ksep_\mf{S}(S)$ to be a single point. Otherwise, define $\ksep_{\mf{S}}(S)$ to be the graph so that:
\begin{itemize}
    \item vertices are all multicurves $x$ on $S$ so that each component of $S \sminus x$ is not an element of~$\mf{S}$;
    \item  two multicurves $x$ and $y$ are joined by an edge if either of the following conditions hold:
\begin{enumerate}
\item $x$ differs from~$y$ by either adding or removing a single curve (see Figure~\ref{figure: path1});
\item $x$ differs from $y$ by ``flipping" a curve in some subsurface of~$S$, that is, $y$ is obtained from $x$ by replacing a curve $\alpha$ by a curve~$\beta$, where $\alpha$ and $\beta$ are contained in the same component $Y_\alpha$ of $S \sminus (x\sminus\alpha)$ and are adjacent in~$\C(Y_\alpha$) (see Figure~\ref{figure: path2}). 
\end{enumerate}
\end{itemize}
If $\G(S)$ is a graph of multicurves, define $\ksep_\G(S) = \ksep_\mf{S}(S)$ where $\mf{S} = \Wit(\G)$. 
\end{definition}

The connectedness of the pants graph implies $\ksep_\mf{S}(S)$ is always connected \cite[Claim 3.3]{vokessep}. There is a natural inclusion $\G(S) \to \ksep_\G(S)$ since every vertex of $\G(S)$ will also be a vertex of $\ksep_\G(S)$. Theorem \ref{theorem: k qi g} below gives sufficient conditions for this map to be a quasi-isometry.

\begin{figure}[h!]
\centering
\begin{subfigure}[b]{0.235\textwidth}
\centering
\includegraphics[width=\textwidth]{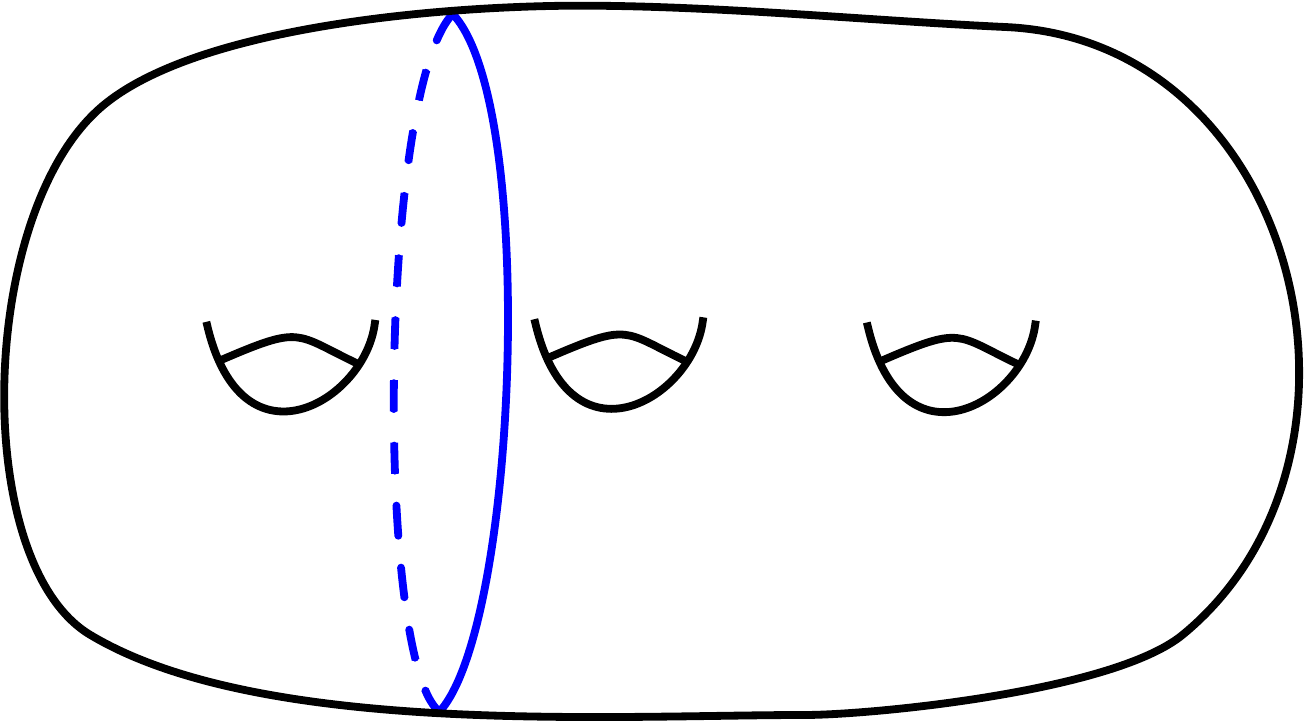}
\end{subfigure}
\begin{subfigure}[b]{0.235\textwidth}
\centering
\includegraphics[width=\textwidth]{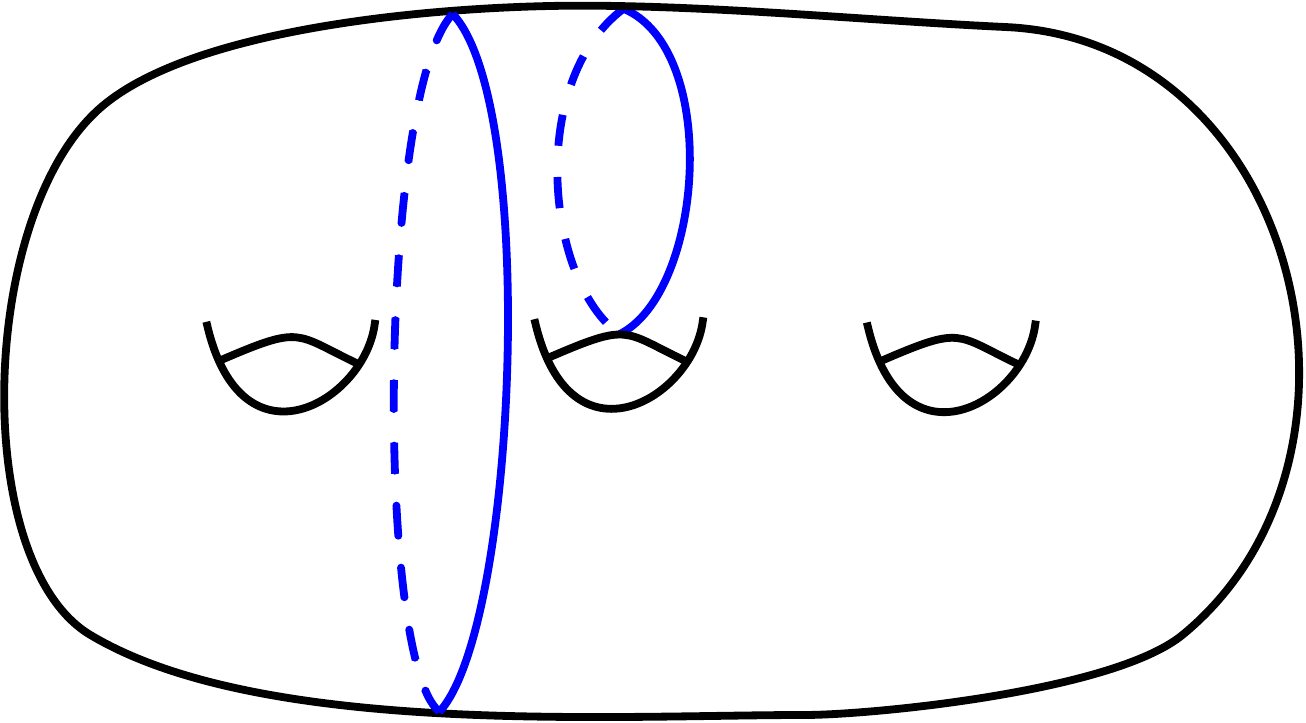}
\end{subfigure}
\begin{subfigure}[b]{0.235\textwidth}
\centering
\includegraphics[width=\textwidth]{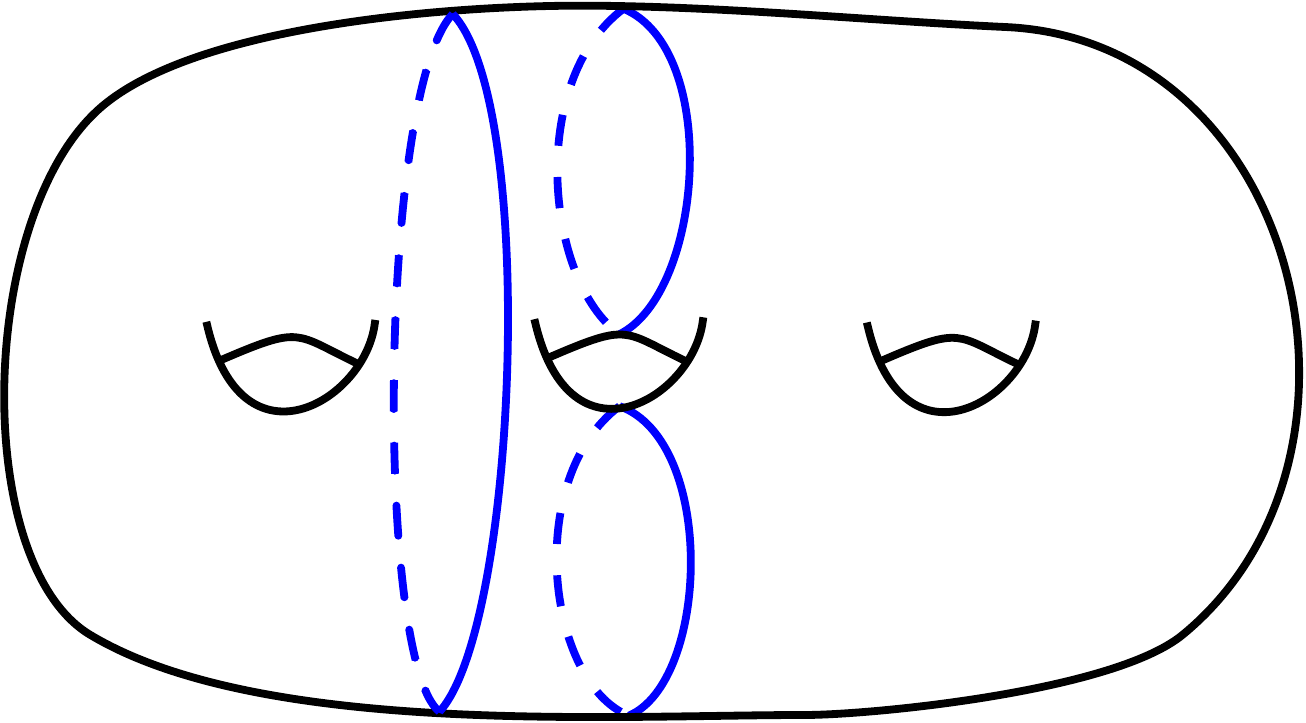}
\end{subfigure}
\begin{subfigure}[b]{0.235\textwidth}
\centering
\includegraphics[width=\textwidth]{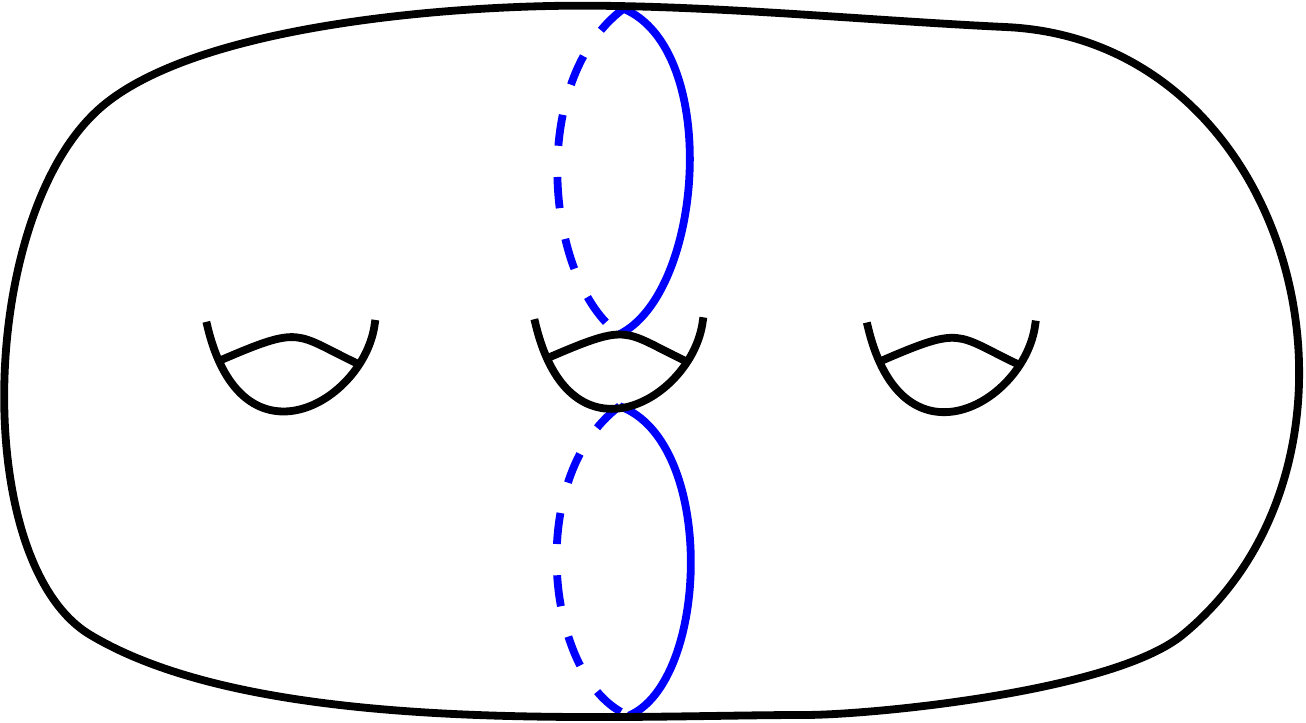}
\end{subfigure}
\caption{An example of a path in $\ksep_{\sep}(S_3)$ given by adding and removing curves.}
\label{figure: path1}
\end{figure}

\begin{figure}[h!]
\centering
\begin{subfigure}[b]{0.25\textwidth}
\centering
\includegraphics[width=\textwidth]{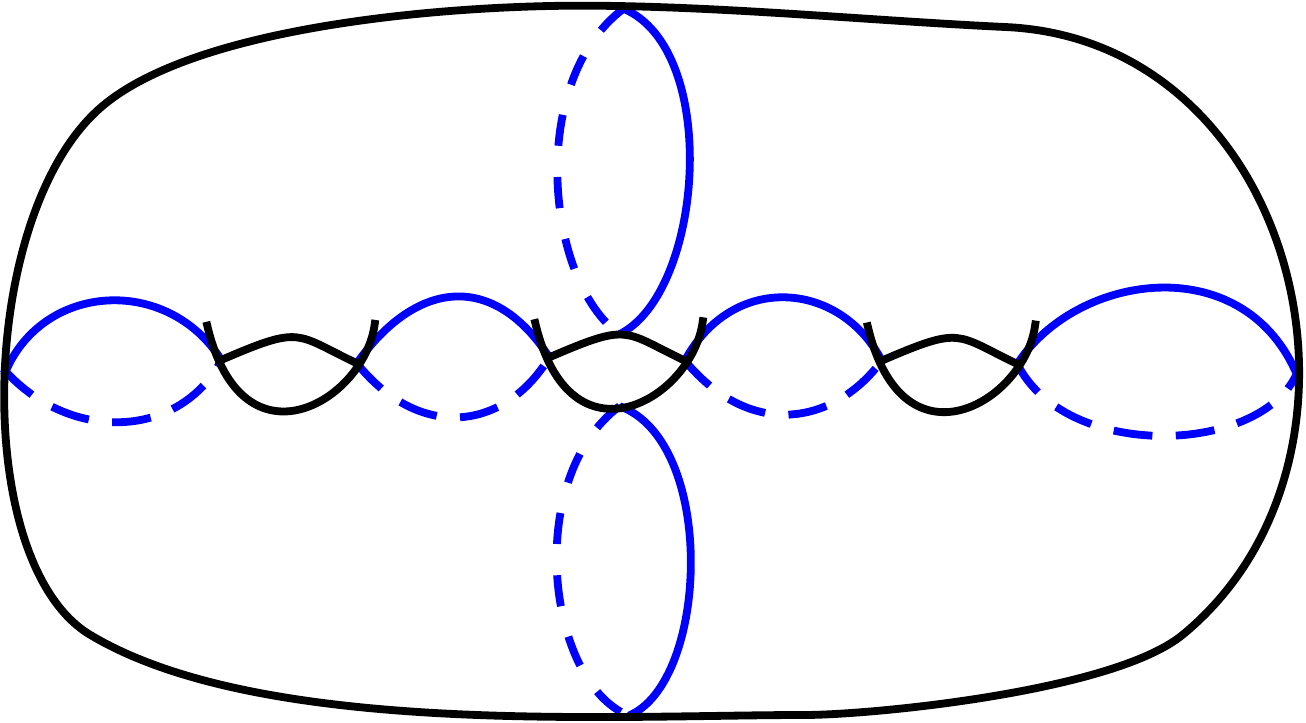}
\end{subfigure}
\qquad \qquad
\begin{subfigure}[b]{0.25\textwidth}
\centering
\includegraphics[width=\textwidth]{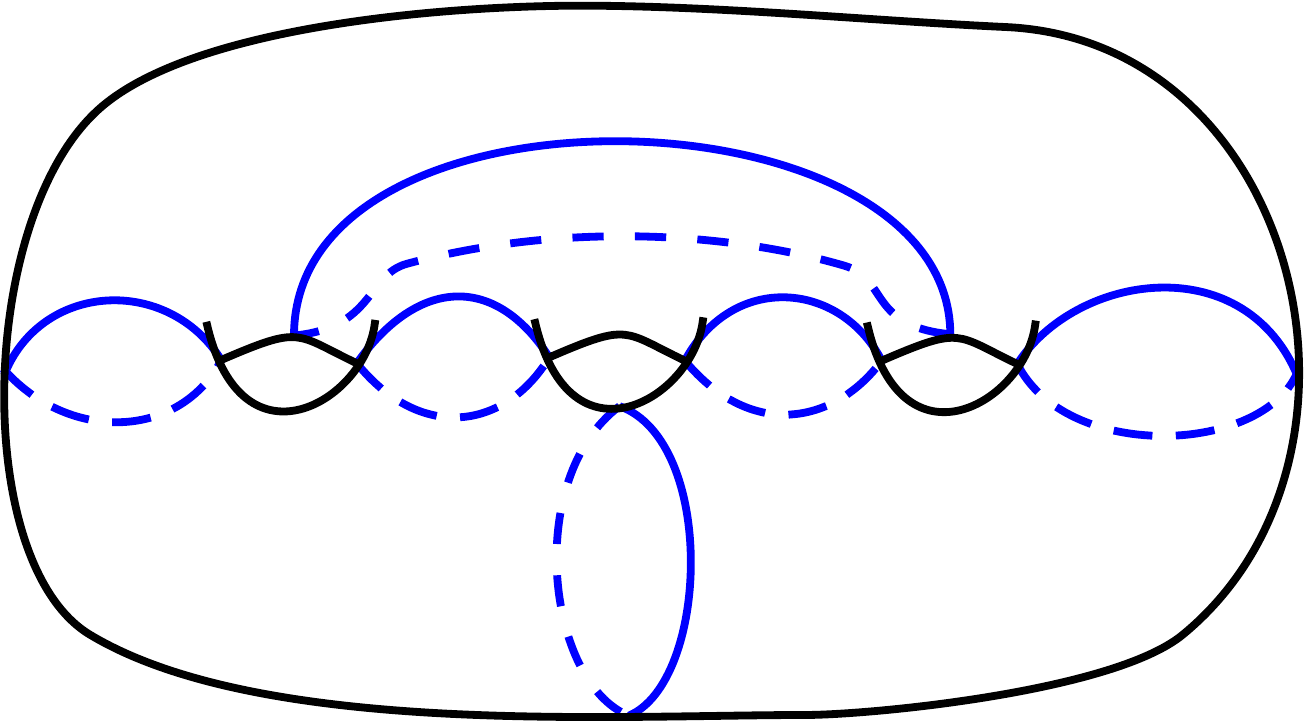}
\end{subfigure}
\caption{An example of a flip move in $\ksep_{\sep}(S_3)$.}
\label{figure: path2}
\end{figure}

The main result of \cite{vokessep} gives simple conditions for a graph of multicurves to be a hierarchically hyperbolic space. We direct the reader to \cite{hhs2,HHS_survey} for a complete definition of a hierarchically hyperbolic space and instead will only note the salient consequences in the context of this paper.

\begin{definition}[Hierarchical graph of multicurves]
We call a graph of multicurves $\mc{G}(S)$  \emph{\hierarchical} if 
\begin{enumerate}
    \item $\mc{G}(S)$ is connected;
    \item the action of the mapping class group on the set of curves on $S$ induces an action by graph automorphisms on $\mc{G}(S)$;
    \item there exists $R>0$ such that any two adjacent vertices of $\mc{G}(S)$ intersect at most $R$ times; 
    \item $\Wit\bigl(\mc{G}(S)\bigr)$ contains no annuli.
\end{enumerate}
\end{definition}

\begin{theorem}[{\cite[Theorem 1.1]{vokessep}}]\label{theorem: hierarchial implies HHS}
If $\G(S)$ is a hierarchical graph of multicurves on the surface $S$, then $\G(S)$ is a hierarchically hyperbolic space.
\end{theorem}

Since  $\sep(S)$ is connected (Theorem~\ref{theorem: sep connected}) and has no annular witnesses (Proposition~\ref{proposition: sep holes}), the definition of the graph ensures it is hierarchical.

The proof of Theorem~\ref{theorem: hierarchial implies HHS} relies on strong connections between $\G(S)$ and $\ksep_\G(S)$ when $\G(S)$ is hierarchical. 

\begin{lemma}[{\cite[Section 3.1]{vokessep}}]\label{lemma:k is hierarchical}
Let $S$ be a surface with positive complexity. If $\G(S)$ is a hierarchical graph of multicurves on $S$, then $\ksep_\G(S)$ is hierarchical and $\Wit\bigl( \ksep_\G(S) \bigr) = \Wit\bigl(\G(S)\bigr)$.
\end{lemma}

\begin{theorem}[{\cite[Proposition 4.1]{vokessep}}] \label{theorem: k qi g}
Let $\G(S)$ be a hierarchical graph of multicurves on~$S$. The inclusion map $\G(S) \to \ksep_\G(S)$ is a quasi-isometry.
\end{theorem}

The most prominent consequence of hierarchical hyperbolicity is a distance formula in the same style as Masur and Minsky's distance formula for the mapping class group~\cite{mm2}.
For any witness $Y$ for $\G(S)$ and vertices $x, y \in \G(S)$, the subsurface projections of $x$ and $y$ to $\C(Y)$ are non\hyp{}empty, so the distance $d_Y(x,y)$ is defined.

\begin{theorem}[Distance formula; {\cite[Corollary 1.2]{vokessep},\cite[Theorem 4.5]{hhs2}}]
\label{theorem: distance formula} \ \\
Let $\G(S)$ be a hierarchical graph of multicurves on $S$ and $\mf{S} = \Wit(S)$. There exists $\sigma_0$ such that for all $\sigma \geq \sigma_0$, there  are $K\geq 1$, $L\geq 0$ so
that for all $x,y\in \G(S)$,
$$ \frac{1}{K}  \sum_{Y\in \mf{S}}\ignore{d_Y(x,y)}{\sigma} - \frac{L}{K} \leq d_{\G}(x,y) \leq K \sum_{Y\in \mf{S}}\ignore{d_Y(x,y)}{\sigma} +L$$
\noindent where $ \ignore{N}{\sigma} = N$ if $N \geq \sigma$ and $0$ otherwise.
\end{theorem}

Masur and Minsky showed that the curve graph of a surface with positive complexity has infinite diameter by proving the pseudo\hyp{}Anosov elements of $\mcg(S)$ act loxodromically on $\C(S)$. As a consequence of this, pseudo\hyp{}Anosov elements have undistorted orbits in any hierarchical graph of multicurves.

\begin{corollary}\label{corollary:pA_are_undistorted}
Let $S$ be a surface of positive complexity and $\G(S)$ be a hierarchical graph of multicurves  on $S$. Let $W$ be a witness for $\G(S)$  and $\phi \in\mcg(S)$ be a partial pseudo\hyp{}Anosov supported on $W$.
For all $x \in \G(S)$, the map $n \mapsto \phi^n(x)$ is a quasi-isometric embedding of $\mathbb{Z}$ into $\G(S)$. In particular, $\G(S)$ has infinite diameter.
\end{corollary}

\begin{proof}
Let $x \in \G(S)$ and $n \in \mathbb{Z}_{>0}$. Without loss of generality, we can restrict to considering the distance between $x$ and $\phi^n(x)$. The upper bound follows from the triangle inequality: \[d_\G\bigl( x, \phi^n(x) \bigr) \leq  \sum\limits_{i=0}^{n-1} d_\G\bigl( \phi^i(x), \phi^{i+1}(x) \bigr) = n\cdot d \bigl( x,\phi(x) \bigr).\]

For the lower bound, \cite[Proposition~4.6]{mm1} provides a $C > 0$ depending only on $S$ so that $d_W\bigl( x, \phi^n(x) \bigr) \geq C \cdot n$ . By the distance formula (Theorem~\ref{theorem: distance formula}), there exist $K\geq 1$, $L\geq 0$ and $\sigma > 0$ so that 
\[ d_\G \bigl( x, \phi^n(x) \bigr) \geq \frac{1}{K} \left(d_W\bigl( x, \phi^n(x) \bigr)-\sigma\right) - \frac{L}{K} \geq \frac{C}{K} \cdot n - \frac{L+\sigma}{K}. \qedhere \]
\end{proof}

All hierarchically hyperbolic spaces come equipped with a system of sub\hyp{}hierarchically hyperbolic spaces called product regions. An advantage to working with $\ksep_\mf{S}(S)$ is that these product regions can be described concretely.

\begin{definition}[Product region of a multicurve]
Let $\mf{S}$ be the set of witnesses for some hierarchical graph of multicurves on $S$. If $m$ is a multicurve on $S$, define $P_\mf{S}(m) = \{ y \in \ksep_\mf{S}(S) : m\subseteq y \}$. We give $P_{\mf{S}}(m)$ the induced metric from $\ksep_\mf{S}(S)$.
\end{definition}

The justification for calling $P_\mf{S}(m)$ a product region lies in the following corollary of the distance formula, which says $P_\mf{S}(m)$ is quasi-isometric to a product of hierarchical graphs of multicurves on the components of $S \sminus m$. 
In the sequel, if $\mf{S}$ is a collection of subsurfaces of $S$ and $Y$ is a subsurface of $S$, then $\mf{S}_Y = \{ Z \in \mf{S} : Z \subseteq Y\}$.

\begin{corollary}\label{corollary:product regions}
Let $\mf{S}$ be the set of witnesses for some hierarchical graph of multicurves on~$S$. If $m$ is a multicurve on $S$ and $S \sminus m = Y_1 \sqcup \dots \sqcup Y_r$, then $P_\mf{S}(m)$ is quasi-isometric to $\prod_{i=1}^r \ksep_{\mf{S}_{Y_i}}(Y_i)$. Moreover, each factor $\ksep_{\mf{S}_{Y_i}}(Y_i)$  has infinite diameter if and only if $Y_i \in \mf{S}$.
\end{corollary}

\begin{proof}
To simplify notation, we will write $\mf{S}_i$ for~$\mf{S}_{Y_i}$.
When $\mf{S}_i$ is non\hyp{}empty, $\ksep_{\mf{S}_{i}}(Y_i)$ is itself a hierarchical graph of multicurves. By definition, $\mf{S}_i =  \emptyset$ if and only if $Y_i \notin \mf{S}$. Thus, $\ksep_{\mf{S}_{i}}(Y_i)$ is defined to be a point when $ Y_i \notin \mf{S}$ and will be infinite diameter whenever $Y_i \in \mf{S}$ by Corollary~\ref{corollary:pA_are_undistorted}. This proves the final clause.

For the quasi-isometry between $P_\mf{S}(m)$ and $\prod_{i=1}^r \ksep_{\mf{S}_{Y_i}}(Y_i)$, let $\iota \colon \prod_{i=1}^r \ksep_{\mf{S}_{Y_i}}(Y_i) \to \ksep_\mf{S}(S)$ be defined by $ \iota(x_1,\dots,x_r) = x_1 \cup \dots \cup x_r \cup \nolinebreak m$.
The image of this map is exactly $P_\mf{S}(m)$.
We shall use the distance formula (Theorem~\ref{theorem: distance formula}) to show that $\iota$ is a quasi\hyp{}isometric embedding.
We use the notation $A \asymp B$ to denote that there exist $K$ and $L$ such that $\left(A-L\right)/K\le B \le K A + L$.
By taking the maximum of the constants involved, we have that if $A \asymp B$ and $A' \asymp B'$ then $A+A' \asymp B+B'$.

Let $\mc{P}=\prod_{i=1}^r \ksep_{\mf{S}_{i}}(Y_i)$, and $\mathbf{x}=(x_1, \dots, x_r)$, $\mathbf{y}=(y_1, \dots, y_r)$ be elements of~$\mc{P}$.
By taking a sufficiently large $\sigma$, we can apply the distance formula to each factor of~$\mc{P}$ to achieve
\[ d_{\mc{P}}(\mathbf{x}, \mathbf{y}) \asymp \sum_{i=1}^r\sum_{Y\in \mf{S_i}}\ignore{d_Y(x_i,y_i)}{\sigma},\]
where the term corresponding to $\mf{S}_i$ is $0$ if $\mf{S}_i$ is empty.

The subset $P_\mf{S}(m)$ has the metric induced from $\ksep_\mf{S}(S)$, so we have
    \[ d_{P_\mf{S}(m)}\bigl(\iota(\mathbf{x}), \iota(\mathbf{y})\bigr) \asymp \sum_{Y\in \mf{S}}\ignore{d_Y(\iota(\mathbf{x}), \iota(\mathbf{y}))}{\sigma}. \]

We will  show  $d_{P_\mf{S}(m)}(\iota(\mathbf{x}),\iota(\mathbf{y})) \asymp d_\mc{P}(\mathbf{x}, \mathbf{y})$, for constants independent of $\mathbf{x}$ and $\mathbf{y}$, by proving there exists a threshold $\sigma$ for the distance formula so that the non\hyp{}zero terms in the right hand side of the distance formula for $P_\mf{S}(m)$ exactly correspond to the non-zero terms in the right hand side of the distance formula for~$\mc{P}$.

If $m$ intersects $W \in \mf{S}$, then
both $\pi_W(\iota(\mathbf{x}))$ and $\pi_W(\iota(\mathbf{y}))$ contain~$\pi_W(m) \neq \emptyset$, and we have that $\diam(\pi_Y(\iota(\mathbf{x})) \cup \pi_Y(\iota(\mathbf{y})))$ is uniformly bounded. By assuming that the threshold~$\sigma$ of the distance formula is larger than this bound, every term in the distance formula for $P_\mf{S}(m)$ corresponding to such a subsurface~$W$ is~0.
Now suppose $m$ does not intersect $W$.
This happens precisely when $W$ is contained in $Y_i$ for some~$i$ (in particular, this $Y_i$ is in~$\mf{S}$).
In this case, $d_W(\iota(\mathbf{x}), \iota(\mathbf{y}))=d_W(x_i, y_i)$.
\end{proof}

\begin{example}[Product Region]\label{example: productregion}
Let $\mf{S} = \Wit\bigl(\sep(S_g)\bigr)$ and let $m$ be the multicurve dividing $S=S_g$ into two $S_{0,g+1}$ components as shown in Figure~\ref{figure:product region example}.
These components are both witnesses by Proposition~\ref{proposition: sep holes}.
Let $Y_1 \sqcup Y_2 = S \sminus m$.  As no proper subsurface of $Y_i$  is a witness, $\mf{S}_{Y_i} = \{ Y_i\}$. By the distance formula (Theorem~\ref{theorem: distance formula}), this implies $\ksep_{\mf{S}_{Y_i}}(Y_i)$ is quasi-isometric to the curve graph $\C(Y_i)$. Thus, $P_{\mf{S}}(m)$ is quasi-isometric to $\C(Y_1) \times \C(Y_2)$. 

\begin{figure}[h!]
    \centering
   \includegraphics[width=4cm]{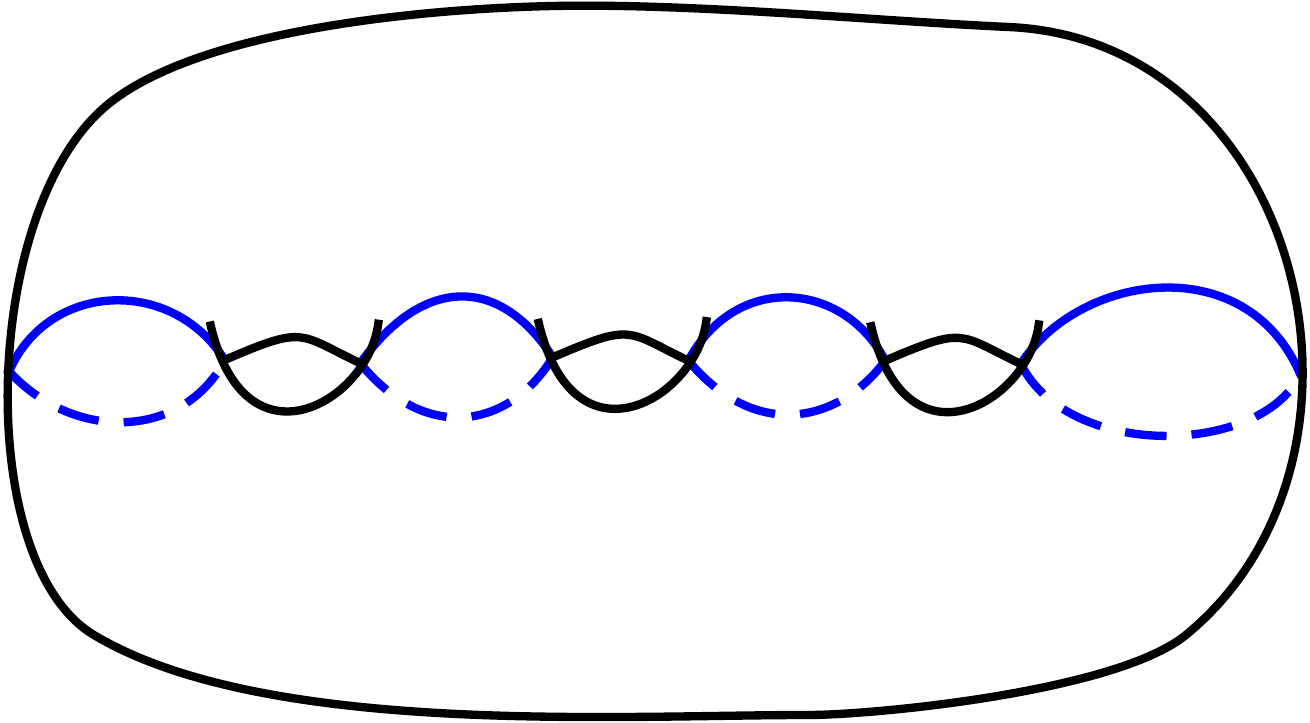}
    \caption{The multicurve $m$ defining $Y_1$ and $Y_2$ in Example~\ref{example: productregion}}. 
    \label{figure:product region example}
\end{figure}
\end{example}

Corollary~\ref{corollary:product regions} tells us that product regions with at least two infinite factors exactly coincide with collections of disjoint witnesses for $\ksep_{\mf{S}}(S)$. As a consequence of hierarchical hyperbolicity, the product regions realize the geometric rank of $\ksep_{\mf{S}}(S)$.
We say that a metric space $X$ has \emph{rank $n$} if $n$ is the largest integer so that there exists a quasi-isometric embedding of $\mathbb{Z}^n$ into $X$.

\begin{corollary}[{\cite[Corollaries 1.4, 1.5]{vokessep}}]\label{cor:rank}
Let $\mf{S}$ be the set of witnesses for some hierarchical graph of multicurves on $S$. The rank of $\ksep_\mf{S}(S)$ is equal to the maximum cardinality of a set of pairwise disjoint elements of $\mf{S}$. Further, $\ksep_\mf{S}(S)$ is hyperbolic if and only if $\mf{S}$ does not contain any pairs of disjoint subsurfaces.

\end{corollary}

\begin{proof}
Let $\{Y_1,\dots,Y_\nu\}$ be a set of  pairwise disjoint elements of $\mf{S}$ with  maximal cardinality and $n$ be the rank of $\ksep_\mf{S}(S)$. By \cite[Corollary 1.4]{vokessep}, $n \leq \nu$. By Corollary \ref{corollary:product regions}, $\prod_{i=1}^\nu \ksep_{\mf{S}_{Y_i}}(Y_i)$ quasi-isometrically embeds into $\ksep_\mf{S}(S)$. Each  $\ksep_{\mf{S}_{Y_i}}(Y_i)$ has rank at least 1 by Corollary \ref{corollary:pA_are_undistorted}, thus $\nu \leq n$. The hyperbolicity statement is precisely \cite[Corollary 1.5]{vokessep}.
\end{proof}

In the case of the separating curve graph, the disjoint witnesses are very restricted and can be described concretely and concisely in all cases.

\begin{proposition} \label{proposition: disjoint witnesses}
Let $S= S_{g,b}$ with $2g+b \geq 4$.  If  $(g,b) \in \{(1,2),(2,0),(2,1)\}$ or $b\ge 3$, then no two witnesses for $\sep(S)$ are disjoint. Otherwise, pairs of disjoint witnesses exist and always have the following form.
\begin{enumerate}
    \item If $b=0$ and $g\ge3$, then a pair of disjoint witnesses is two copies of~$S_{0, g+1}$ that meet along all of their boundary components (Figure~\ref{figure: disjointholes1}).
    \item If $b=2$ and $g\ge2$, then a pair of disjoint witnesses is two copies of~$S_{0, g+2}$ that meet along all but one of their boundary components (Figure~\ref{figure: disjointholes3}).
    \item If $b=1$ and $g\ge3$, then a pair of disjoint witnesses is either: 
    \begin{itemize}
        \item a copy of $S_{0,g+1}$ and a copy of $S_{0,g+2}$ that meet along all the boundary components of the $S_{0,g+1}$ subsurface (Figure~\ref{figure: disjointholes2a});
        \item two copies of $S_{0,g+1}$ that meet along all but one of their boundary components, so that neither contains $\partial S$ (Figure~\ref{figure: disjointholes2}).
    \end{itemize}
\end{enumerate}
In particular, the rank of $\sep(S)$ is~1 if $(g,b) \in \{(1,2),(2,0),(2,1)\}$ or $b \geq 3$, and 2 otherwise.
\end{proposition}

\begin{figure}[h!]
\centering
\begin{subfigure}[b]{0.22\textwidth}
\centering
\includegraphics[width=\textwidth]{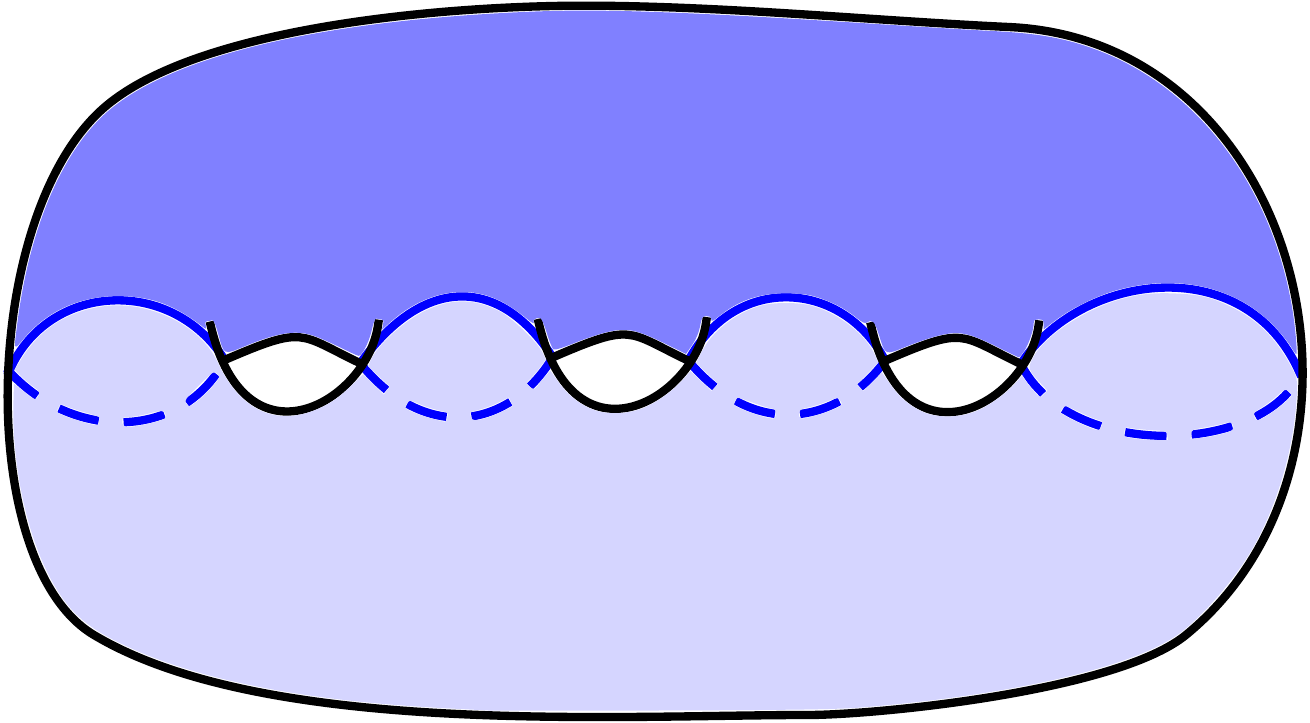}
\subcaption{}
\label{figure: disjointholes1}
\end{subfigure}
\enspace
\begin{subfigure}[b]{0.22\textwidth}
\centering
\includegraphics[width=\textwidth]{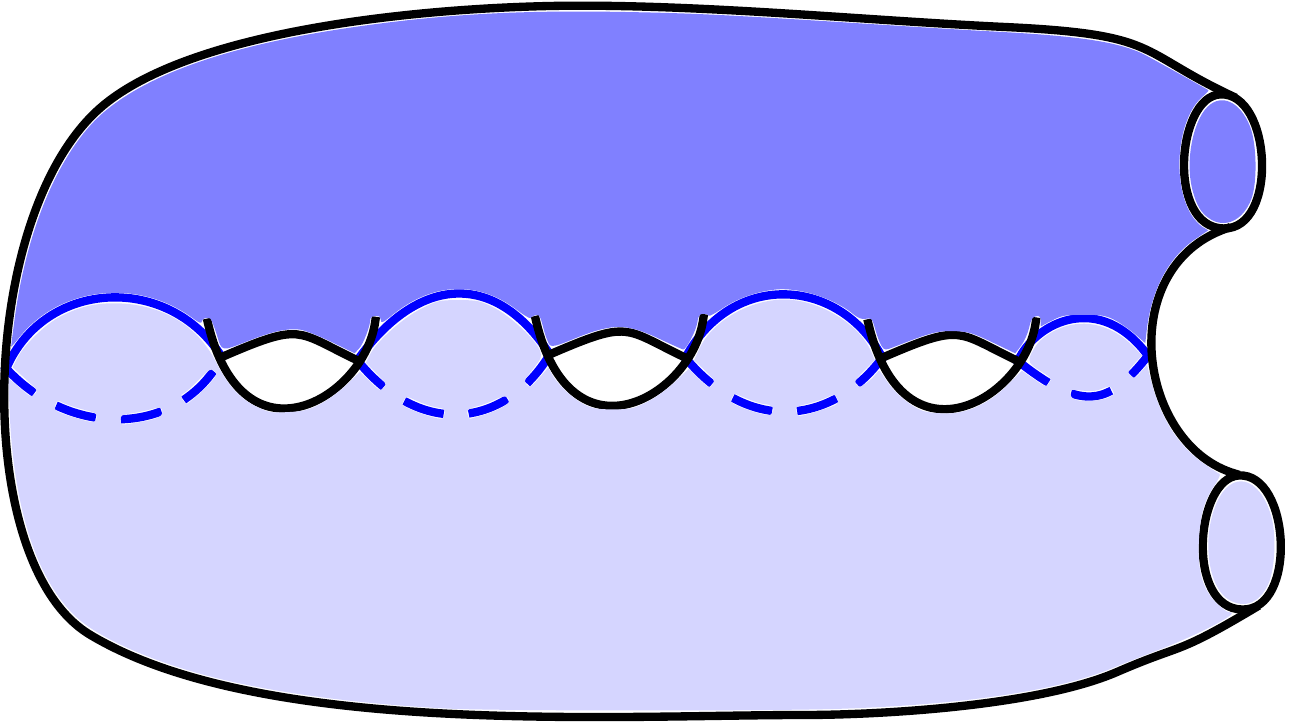}
\subcaption{}
\label{figure: disjointholes3}
\end{subfigure}
\enspace
\begin{subfigure}[b]{0.23\textwidth}
\centering
\includegraphics[width=\textwidth]{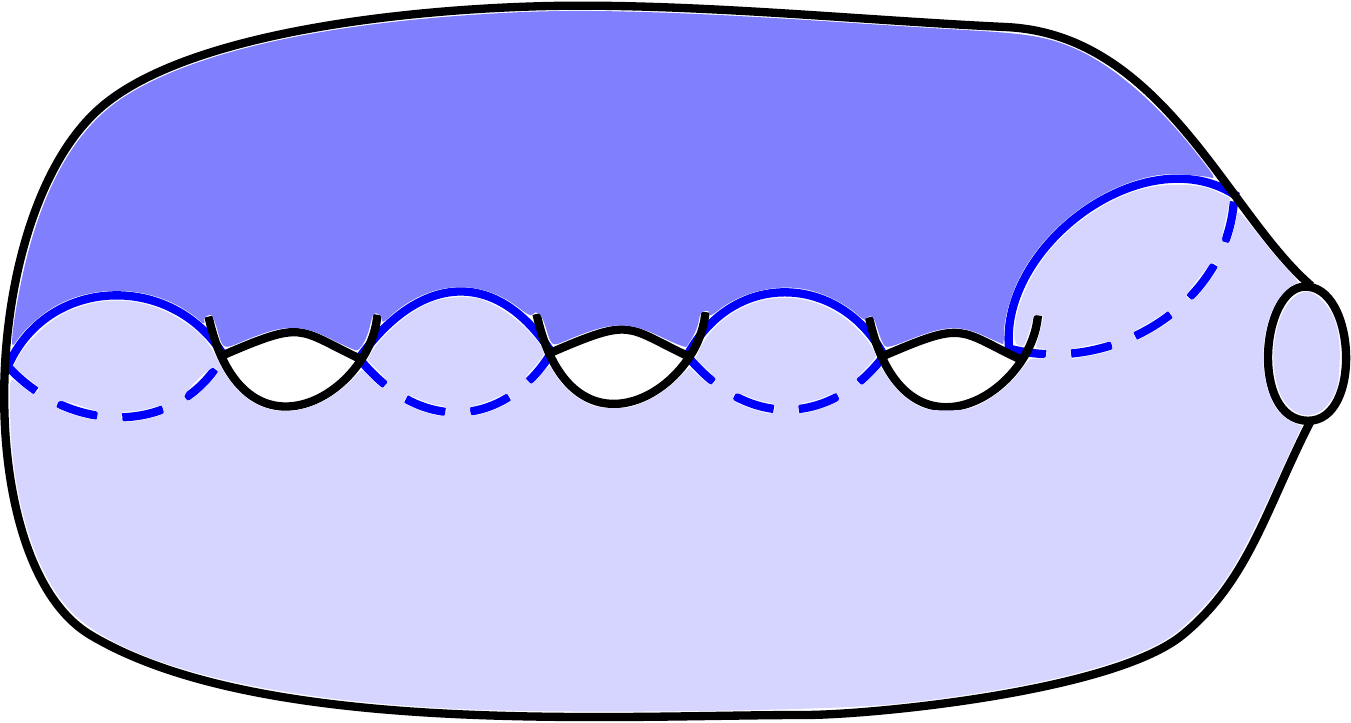}
\subcaption{}
\label{figure: disjointholes2a}
\end{subfigure}
\enspace
\begin{subfigure}[b]{0.23\textwidth}
\centering
\includegraphics[width=\textwidth]{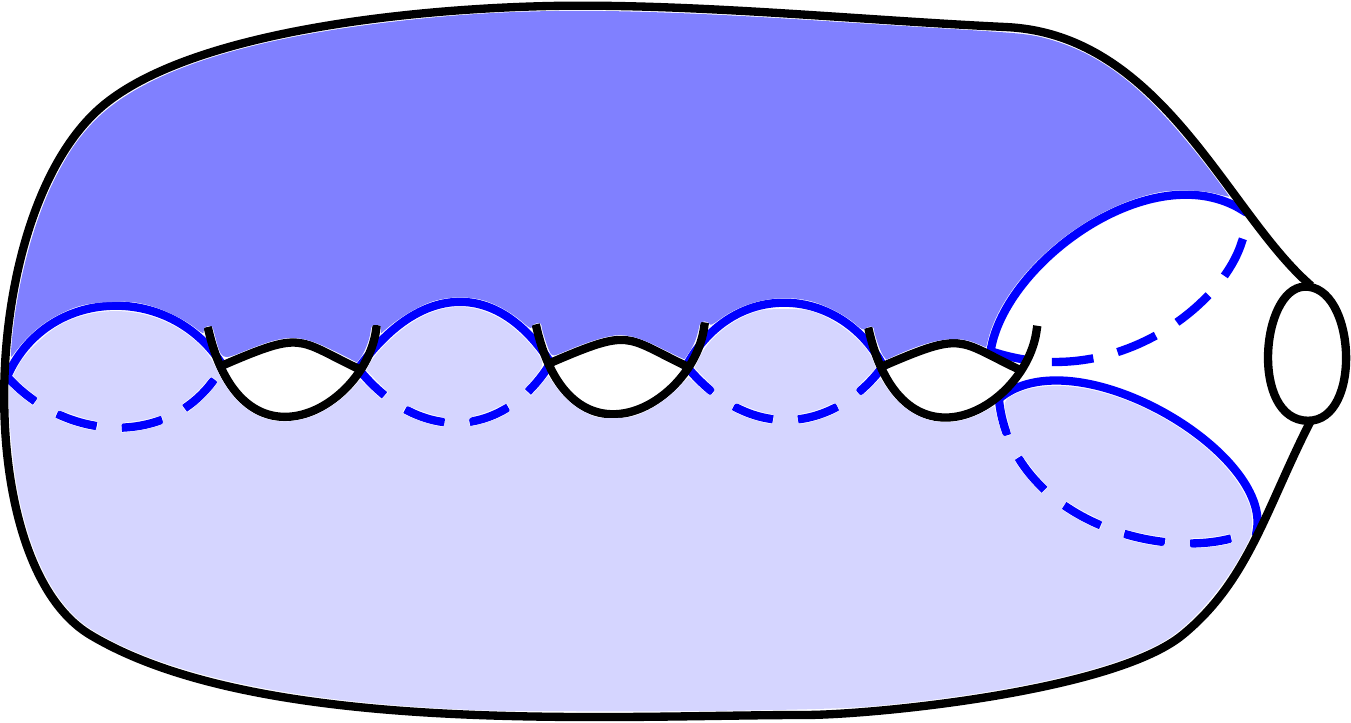}
\subcaption{}
\label{figure: disjointholes2}
\end{subfigure}
\caption{The possibilities for pairs of disjoint witnesses for $\sep(S)$ for $S_{3,0}$, $S_{3, 2}$ and $S_{3,1}$ up to the action of the mapping class group.}
\label{figure: disjointholes}
\end{figure}

\begin{proof}

Recall, Proposition~\ref{proposition: sep holes} established that $\Wit\bigl(\sep(S)\bigr)$ is the set of all positive complexity connected subsurfaces $W$ so that each component of $S\sminus W$ is planar and contains at most one boundary component of $S$.

First suppose that~$S=S_{g,b}$ with $(g,b) \in \{(1,2),(2,0),(2,1)\}$.
For $S_{1,2}$, there are no pairs of disjoint, positive complexity subsurfaces.
For the other two cases, if two positive complexity subsurfaces are disjoint, they must both contain genus.
Since each of the subsurfaces is contained in the complement of the other, neither can be a witness for~$\sep(S)$.

Now suppose that $S$ has at least three boundary components.
Suppose $W,Y \in \Wit\bigl( \sep(S) \bigr)$ are disjoint.
Since any connected subsurface containing $Y$ is also a witness for $\sep(S)$, we can assume that $Y$ is a component of~$S \sminus W$.
In particular,  $S \sminus Y$ is connected and $Y$ is a planar subsurface that contains at most one boundary component of~$S$.
Since $b \geq 3$, this implies $S \sminus Y$ contains at least  two boundary components of~$S$.
However, this contradicts that $Y$ is in $\Wit\bigl(\sep(S)\bigr)$, and hence, we have that $\Wit\bigl( \sep(S) \bigr)$ contains no pairs of disjoint subsurfaces.

Now suppose that~$S=S_{g,b}$, with $b\le2$ and $(g,b) \not \in \{(1,2),(2,0),(2,1)\}$.
Let $W$ and~$Y$ be disjoint subsurfaces in $\Wit\bigl(\sep(S)\bigr)$.
Assume that $W$ is \emph{minimal}, in the sense that no proper subsurface of $W$ is a witness for $\sep(S)$, and assume that $Y$ is a component of~$S \sminus W$. We claim $S \sminus W$ is connected, implying $S \sminus W = Y$.

Suppose that~$S \sminus W$ is disconnected, and let~$Z$ be a component of~$S \sminus W$ that is not~$Y$.
If $\partial Z \cap \partial W$ contains more than one curve, then $S \sminus Y$ has positive genus, which contradicts that $Y$ is in $\Wit\bigl(\sep(S)\bigr)$.
Hence, $Z$ meets $W$ along a single curve.
{However, we also know that $Z$ is planar and contains at most one component of~$\partial S$.}
Thus,~$Z$ must be a disk or a peripheral annulus, contradicting that~$W$ is essential.
Therefore, $S \sminus W$ is connected and equal to $Y$.

Since $W$ and $Y$ are both witnesses and $S \sminus W = Y$, each of~$W$ and~$Y$ must be planar and must contain at most one component of~$\partial S$. Since we are assuming that $W$ is minimal,  we have precisely one possibility for what $W$ and $Y$ can be for each of $S_{g,0}$, $S_{g,1}$ and~$S_{g,2}$:
\begin{itemize}
    \item if $S=S_{g,0}$, then $W$ and $Y$ are both copies of~$S_{0, g+1}$ that meet along all of their boundary components (Figure~\ref{figure: disjointholes1});
    \item if $S=S_{g, 2}$, then $W$ and $Y$ are both copies of~$S_{0, g+2}$ that meet along all but one of their boundary components (Figure~\ref{figure: disjointholes3});
    \item if $S=S_{g,1}$, then $W$ is a copy of~$S_{0,g+1}$ and $Y$ is a copy of~$S_{0,g+2}$ with $W$ and $Y$ meeting along all boundary components of $W$ (Figure~\ref{figure: disjointholes2a}).
\end{itemize}

The subsurfaces $W$ and $Y$ described above are indeed witnesses for~$\sep(S)$ as they will always have positive complexity for the surfaces $S$ we are considering.

When $b=0$ or $b=2$, there is an element of $\mcg(S)$ that maps $W$ to $Y$. Since we assumed that no proper subsurface of $W$ is in $\Wit\bigl(\sep(S)\bigr)$, the same is true of~$Y$. 
In particular,  the maximal cardinality of a set of pairwise disjoint elements of $\Wit\bigl(\sep(S)\bigr)$ is~2, and, up to the action of $\mcg(S)$, there is a unique pair of disjoint witnesses for $\sep(S)$. 

In the case of $S_{g,1}$, $Y$ is not minimal as it contains a subsurface $W'$ that is the image of $W$ under some element of $\mcg(S)$, and hence a witness for $\sep(S)$.
We claim the only subsurfaces of $Y$ that are witnesses for $\sep(S)$ are obtained by removing a single pair of pants containing $\partial S$ from $Y$.

Let $Z$ be a proper connected subsurface of~$Y$. Since $Y$ is planar, $Z$ must also be planar, and can be obtained from $Y$ by successively removing copies of $S_{0,3}$. Let $P$ be the first copy of $S_{0,3}$ removed from $Y$ to move towards $Z$. If $P$ does not contain the boundary component of $S$, then $\partial P$ contains two disjoint curves of $\partial W$. This implies that $S \sminus Z$ has positive genus, and hence that $Z$ is not a witness for $\sep(S)$.  If $P$ does contain the boundary component of $S$, then $Y \sminus P$ is in the mapping class group orbit of $W$ and hence minimal. Therefore, a subsurface $Z$ of~$Y$ is an element of $\Wit\bigl(\sep(S_{g,1})\bigr)$ if and only if it is obtained from $Y$ by removing a single pair of pants containing ~$\partial S$.
Since a copy of $S_{0,3}$ is never a witness, we again find that sets of pairwise disjoint witnesses for $\sep(S_{g,1})$ have cardinality at most~2. Moreover, up to the action of $\mcg(S)$, there are exactly the two possibilities stated in the proposition.
\end{proof}

The restrictions on the disjoint witnesses for $\sep(S)$ have direct implications for the (relative) hyperbolicity of the graph. The second author concluded hyperbolicity in  the case of $b\geq 3$ or $(g,b) \in \{(1,2),(2,0),(2,1)\}$ from the lack of disjoint witnesses~\cite{vokessep}, and the unique form of the disjoint witnesses in the case where $b =0$ or $b=2$ is an essential component of the first author's proof of relative hyperbolicity~\cite{russell}.

\section{Thickness of the Separating Curve Graph} \label{section: thickness}

We now use the hierarchy structure and the capping map to show that $\sep(S_{g,1})$ is a thick metric space when $g \geq 3$. For the remainder of this section, let $S = S_{g,1}$ with $g \geq 3$, and let $\Sigma = S_{g,0}$ be the surface obtained from $S$ by capping off the boundary component with a disk. For convenience, we shall use $\mf{X}$ to denote the set of witnesses for $\sep(S)$, $\ksep(S)$ to denote $\ksep_\mf{X}(S)$, and $P(m)$ to denote $P_\mf{X}(m)$ for any multicurve $m$ on $S$. If $A$ is a subset of a metric space $X$, then $\mc{N}_C(A)$ will denote the $C$\hyp{}neighborhood of $A$ in~$X$.

\begin{definition}[Thick metric space]
A metric space $X$ is \emph{thick of order~0} if none of its asymptotic cones have cut points. In particular,  $X$ will be thick of order~0 if $X$ is quasi-isometric to a product of two infinite diameter metric spaces.

A metric space $X$ is \emph{thick of order at most~$n$} if there exists a constant $C\geq 0$ and a collection of subsets $\{P_\alpha\}_{\alpha \in I}$ such that the following hold.

\begin{itemize}
    \item (Thickness) Each $P_\alpha$ is thick of order at most~$n-1$.
    \item (Coarsely Covering)  The space $X$ is contained in the $C$\hyp{}neighborhood of $\bigcup_{\alpha \in I} P_\alpha$.
    \item (Thick Chaining) For any $P_\alpha$ and $P_{\alpha'}$ there exists a sequence $$P_\alpha = P_0, P_1, \dots, P_k = P_{\alpha'}$$ such that $\mc{N}_C(P_i) \cap \mc{N}_C(P_{i+1}) $ has infinite diameter for all $0\leq i \leq k-1$.
\end{itemize}
\end{definition}

Behrstock, Dru\c{t}u, and Mosher established that if $X$ is thick of any order, then $X$ cannot be relatively hyperbolic \cite[Corollary~7.9]{Behrstock_Drutu_Mosher_Thickness}. Since  the order of thickness is a quasi-isometry invariant \cite[Remark~7.2]{Behrstock_Drutu_Mosher_Thickness}, to prove that $\sep(S)$ is thick of order at most~2, it will be sufficient to prove the same for $\ksep(S)$.

\begin{theorem}\label{theorem: sep is thick}
For all $g \geq 3$, $\sep(S_{g,1})$ is thick of order at most~$2$. In particular, $\sep(S_{g,1})$ is not relatively hyperbolic.
\end{theorem}

To show that $\ksep(S)$ is thick of order at most~$2$, we need to build thick of order at most~1 subsets that coarsely cover $\ksep(S)$ and that can be thickly chained together.  The thick of order at most~1 subsets in turn need to be built from thick of order~0 subsets which can be thickly chained together. These thick of order~0 pieces will come from the product regions corresponding to pairs of disjoint witnesses for $\sep(S)$. 

\begin{lemma}\label{lem:product_regions_are_thick}
Let $W$ and $Y$ be a pair of disjoint witnesses for $\sep(S)$. If $m= \partial W \cup \partial Y$, then $P(m)$ is quasi\hyp{}isometric to a product of two infinite diameter metric spaces, and hence thick of order~0.
\end{lemma}

\begin{proof}
By the classification of pairs of disjoint witnesses for $\sep(S)$ given in Proposition~\ref{proposition: disjoint witnesses}(3), the only components of $S\sminus m$ that are witnesses are $W$ and~$Y$.
Hence, by Corollary~\ref{corollary:product regions}, $P(m)$ is quasi-isometric to $\ksep_{\X_W}(W) \times \ksep_{\X_Y}(Y)$. Since $W,Y \in \Wit\bigl(\sep(S)\bigr)$, $\ksep_{\X_W}(W)$ and $\ksep_{\X_Y}(Y)$ are both infinite diameter.
\end{proof}

To understand how the product regions chain together, we use the following graph, and show the edge relation encodes the thick intersection between product regions.

\begin{definition}[Graph of disjoint witnesses]
Let $\DW$ be the graph whose vertices are multicurves $m$ on $S$ so that $S \sminus m$ contains a pair of disjoint witnesses for $\sep(S)$. Two multicurves $m,n \in \DW$ are joined by an edge if $n$ can be obtained from $m$ by adding or deleting a single curve.
\end{definition}

By Proposition~\ref{proposition: disjoint witnesses}(3), the vertices of $\DW$ come in two types up to homeomorphism (see Figure~\ref{figure: vertices of dw}).

\begin{figure}[h!]
\centering
\begin{subfigure}{0.23\textwidth}
\centering
\includegraphics[width=\textwidth]{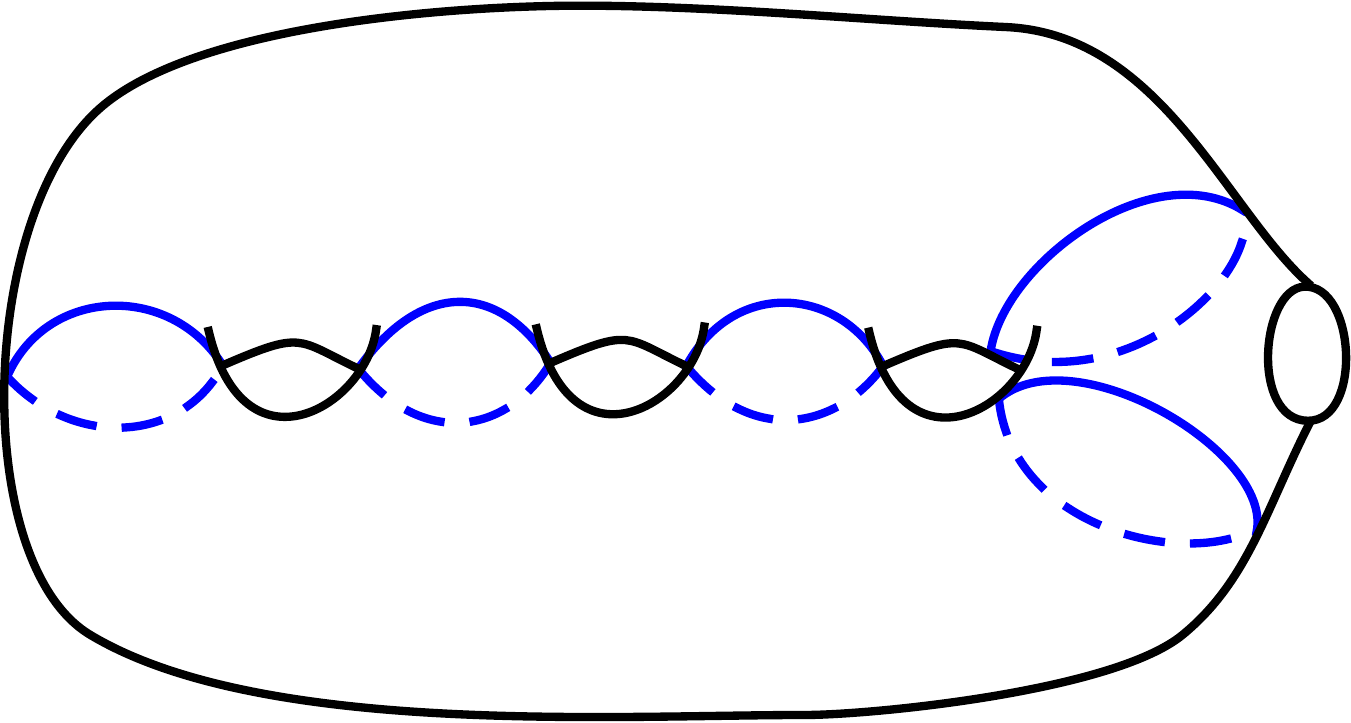}
\subcaption{}

\end{subfigure}
\hspace{1cm}
\begin{subfigure}{0.23\textwidth}
\centering
\includegraphics[width=\textwidth]{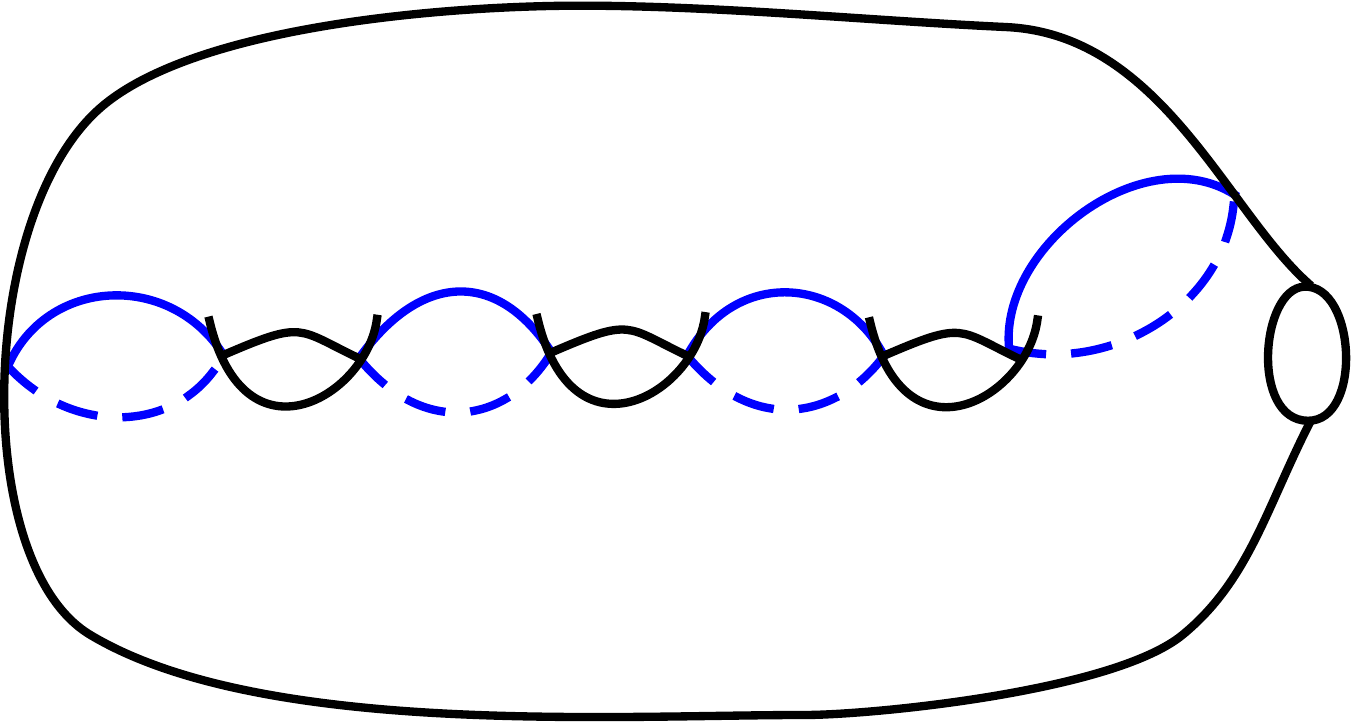}
\subcaption{}

\end{subfigure}
\caption{The two possible topological types of multicurves in $\DW$, in the genus 3 case.}
\label{figure: vertices of dw}
\end{figure}

\begin{lemma}\label{lem:order_zero_large_intersection}
If $m,n \in \DW$ are connected by an edge, then $P(m) \cap P(n)$ is equal to $P(m \cup n)$ and is an  infinite diameter subset of $\ksep(S)$.
\end{lemma}

\begin{proof}
That $P(m) \cap P(n) = P(m \cup n)$ follows from the definition of the product regions.
Without loss of generality, $m \subseteq n$, so $m \cup n=n$ and  $P(m \cup n) = P(n)$, which has infinite diameter by Lemma~\ref{lem:product_regions_are_thick}.
\end{proof}

In light of Lemma~\ref{lem:order_zero_large_intersection}, if $\Omega$ is a connected component of $\DW$, then the union of all product regions for multicurves in $\Omega$ will be a thick of order~1 subset of~$\ksep(S)$. 
We show that the connected components of $\DW$ (and hence the thick of order~1 subsets) are naturally encoded by the fiber of the map induced by capping the boundary component on~$S$. As an aid to the reader, we will always use Roman letters to denote multicurves on $S$ and Greek letters to denote multicurves on~$\Sigma$.

\begin{definition}[The capping  map]
Recall, $F \colon \C(S) \to \C(\Sigma)$ is the map induced by capping the boundary component of $S$ with a disk.  Let $\FO$ denote the set of multicurves $\mu$ on $\Sigma$ such that there exists $m \in \DW$ with $\mu = F(m)$. For a multicurve $\mu \in \FO$, we shall abuse notation and use $F^{-1}(\mu)$ to denote the full subgraph of $\DW$ spanned by the set $\{m \in \DW : F(m) = \mu\}$.
\end{definition}

\begin{proposition}
\label{prop:fibers_are_connected_components}
The map $\mu \mapsto F^{-1}(\mu)$ is a bijection from $\FO$ to the set of connected components of $\DW$.
\end{proposition}

\begin{proof} 
We first show that each connected component of $\DW$ is contained in the fiber of a single multicurve $\mu \in \FO$.

\begin{claim} \label{claim:components_are_in_fiber}
If $ m,n \in \DW$ are connected by an edge, then $F(m)$ is isotopic to $F(n)$ on~$\Sigma$.
\end{claim}

\begin{proof}
Without loss of generality, assume that $n$ is obtained from $m$ by adding a curve~$a$. From the classification of topological types of vertices of~$\DW$, there exists a curve $b$ in $m$ such that $a$ and $b$ bound a pair of pants with the boundary component of~$S$. Thus $F(a)$ and $F(b)$ are isotopic on $\Sigma$, which implies $F(m)$ is isotopic to $F(n)$ on~$\Sigma$.
\end{proof}

Inductively, Claim~\ref{claim:components_are_in_fiber} implies that every element of a connected component of $\DW$ has the same image under~$F$. The next claim establishes that the connected components are in one-to-one correspondence with the pre-images of the points in~$\FO$.

\begin{claim}\label{claim:fibers_are_connected}
For all $\mu \in \FO$, $F^{-1}(\mu) \subseteq \DW$ is connected.
\end{claim}

\begin{proof}
Let $m, n \in F^{-1}(\mu)$ with $m \neq n$.

First suppose $i(m,n) = 0$, and consider the multicurve $m \cup n$.
The multicurve $F(m)=F(n)=F(m \cup n)$ on $\Sigma$ contains $g+1$ curves. Therefore, there is a multicurve $m'$ on $S$ with exactly $g+1$ curves from $m \cup n$ so that $F(m')=F(m \cup n)$. Two disjoint, non-isotopic curves on $S$ have the same image under $F$ if and only if they cobound a pair of pants with $\partial S$. Thus, any curve in $(m \cup n) \sminus m'$ must cobound a pair of pants with $\partial S$ and a curve of~$m'$. Hence, there can be at most one curve of $m \cup n$ that is not in $m'$ (and since $m \neq n$, such a curve exists).
We therefore have that $S \sminus (m \cup n)$ contains two witnesses homeomorphic to $S_{0, g+1}$, and a copy of~$S_{0,3}$.
This implies $m$ and $n$ are connected in~$\DW$ as the multicurve $m \cup n$ is a vertex of $\DW$ that is either adjacent to or equal to each of $m$ and $n$.

Now suppose $i(m,n) = k$ and assume, by induction, that any two vertices of $ F^{-1}(\mu)$ with intersection number less than $k$ can be connected within~$F^{-1}(\mu)$.
Since $F(m) = F(n)$, any intersections between $m$ and $n$ come in pairs and form a 1\hyp{}holed bigon containing the boundary component of~$S$.
Let $a \in m$ and $b \in n$ be curves whose intersection forms the innermost bigon around the boundary component of ~$S$.
We perform the surgery shown in Figure~\ref{figure:surgery} across the boundary component of~$S$ to produce a curve $b'$ that is disjoint from~$n$.
Let $n'= (n \sminus b) \cup b'$. Then $i(m,n') \leq k-2$, $i(n,n') =0$, and $F(n') = F(n) = F(m)$. Thus, by induction, $m$ is connected to $n$ in~$F^{-1}(\mu)$.
\end{proof}

\begin{figure}[h!]
    \centering
    \def\svgscale{1} 
    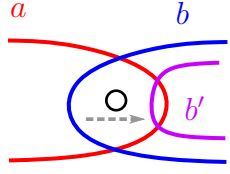
    \caption{Surgering $b$ to $b'$ across the boundary component of $S$. }
    \label{figure:surgery}
\end{figure}

Combining Claims \ref{claim:components_are_in_fiber} and \ref{claim:fibers_are_connected}, we have that $\mu \mapsto F^{-1}(\mu)$ is a bijection from $\FO$ to the set of connected components of $\DW$.
\end{proof}

We can now describe the thick of order~1 subsets we need to prove $\ksep(S)$ is thick of order at most~2. Each of these subsets is the union of the product regions of all multicurves in a connected component of $\DW$. By Proposition~\ref{prop:fibers_are_connected_components}, these subsets are in correspondence with elements of $\FO$. This correspondence will be critical in thickly chaining any two of these subsets together.

\begin{definition}[Thick of order 1 pieces]
For each $\mu \in \FO$, define \[\mc{X}(\mu) = \displaystyle \bigcup \limits_{m \in F^{-1}(\mu)} P(m).\]
\end{definition}

\begin{lemma}\label{lem:thick_of_order_1_pieces}
For each $\mu \in \FO$, $\mc{X}(\mu) \subseteq \ksep(S)$ is thick of order at most~1. 
\end{lemma}

\begin{proof}
Let $\mu \in \FO$. By construction, $\mc{X}(\mu)$ is covered by the thick of order~0 subsets $\mc{P}=\{P(m) : m \in F^{-1}(\mu)\}$. By Proposition~\ref{prop:fibers_are_connected_components}, the fiber $F^{-1}(\mu)$ is a connected component of $\DW$. If $m,m' \in \DW$ are joined by an edge, then $P(m) \cap P(m') = P(m \cup m')$  and is infinite diameter by Lemma~\ref{lem:order_zero_large_intersection}. So given any two multicurves $m, m' \in F^{-1}(\mu)$, the path from $m$ to $m'$ in $\DW$ gives a thick chain of elements of $\mc{P}$ connecting $P(m)$ to~$P(m')$.
\end{proof}

The next task is to be able to thickly chain any two of the $\mc{X}(\mu)$'s together. As is the case with product regions, we can encode the thick intersection of the $\mc{X}(\mu)$'s using a graph. In this case, the graph has vertex set $\FO$ with an edge between two vertices if they intersect a bounded number of times on~$\Sigma$.

\begin{definition}
Define $\F$ to be the graph with vertex set $\FO$ where two vertices $\mu$ and $\nu$  are joined by an edge if $i(\mu, \nu) \leq 4$. 
\end{definition}

\begin{proposition}\label{prop:intersection_thick_of_order_2}
There exists $C \geq 0$, so that if $\mu,\mu' \in \F$ are joined by an edge, then  $\mc{N}_C\bigl( \mc{X}(\mu) \bigr) \cap \mc{N}_C\bigl(\mc{X}(\mu')\bigr)$ has infinite diameter.
\end{proposition}

\begin{proof}

Let $\mu$ and $\mu'$ be two adjacent vertices of $\F$.
Since $i(\mu,\mu') \le 4$, we can choose $m \in F^{-1}(\mu)$ and $m' \in F^{-1}(\mu')$ so that $i(m,m') \le 4$.
Given a subsurface $Y$ of~$S$, we will call a curve $\alpha$ of~$Y$ \emph{admissible} if $\alpha$ is essential and non\hyp{}peripheral in~$Y$ and no component of $Y \sminus \alpha$ is a 3\hyp{}holed sphere containing~$\partial S$.
We can extend any vertex $v$ of  $\DW$  to an element $z$ of $P(v) \subseteq \ksep(S)$  by adding an admissible curve in each witness component of $S \sminus v$.
We claim that we can extend $m$ and $m'$ to $x \in P(m)$ and $x' \in P(m')$ respectively, so that $i(x, x') \le 20$.

Let $Y$ be a component of $S \sminus m$ that is not a $3$\hyp{}holed sphere. By Proposition \ref{proposition: sep holes}, $Y$ is a sphere with either $g+1$ or $g+2$ boundary components.

\begin{claim}\label{claim:extending DW}
$Y$ contains an admissible curve that intersects $m' \cap Y$ at most twice.
\end{claim}

\begin{proof}
Note, a curve $\alpha$ on $Y$ is admissible if and only if each component of $Y \sminus \alpha$ contains at least two boundary components of $\partial_S Y$.

The intersection of $m'$ with $Y$ is a collection of at most two disjoint arcs and possibly some curves. We may assume there are two arcs, since having fewer can only decrease the number of intersections in what follows. Let $a_1$ and $a_2$ be the arcs of $m' \cap Y$. If either $a_1$ or $a_2$ has endpoints on two different components of $\partial_S Y$,  then the standard surgery shown in Figure~\ref{figure: standardsurgery} will yield an admissible curve on $Y$ that intersects $m'$ at most twice. Assume hence that each of $a_1$ and $a_2$ has both endpoints on a single curve. 

\begin{figure}[h!]
\centering
\begin{subfigure}[b]{0.3\textwidth}
\centering
\def\svgscale{.5}
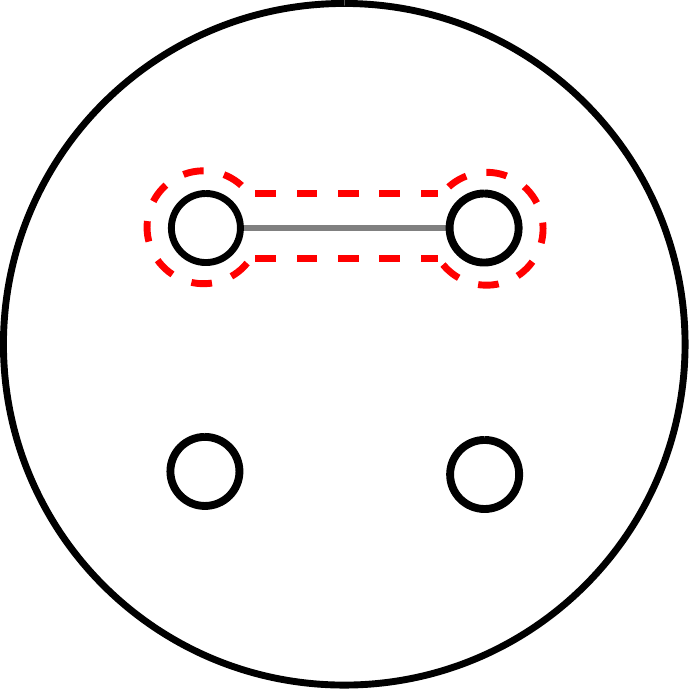
\subcaption{}
\label{figure: standardsurgery}
\end{subfigure}
\enspace
\begin{subfigure}[b]{0.3\textwidth}
\centering
\def\svgscale{.5}
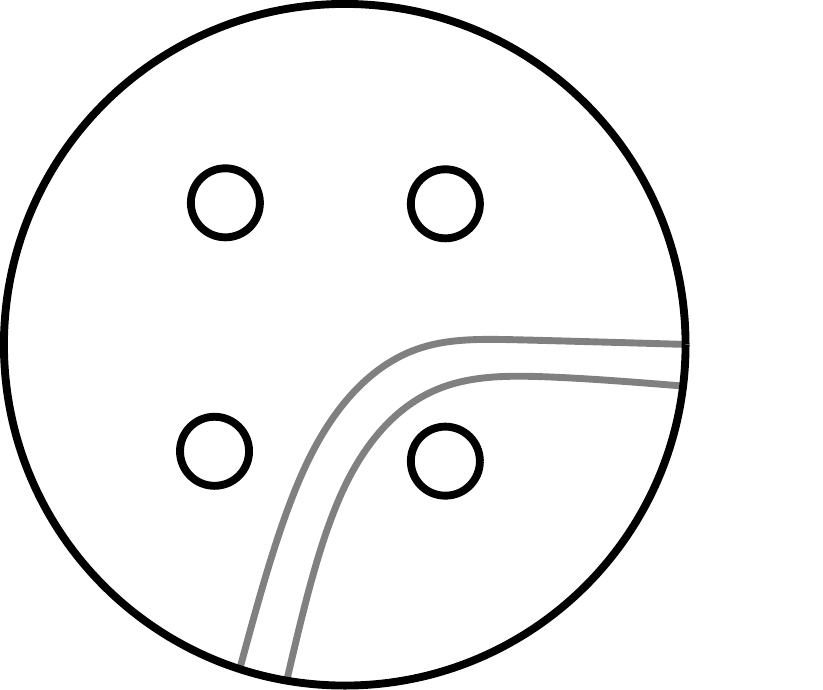
\subcaption{}
\label{figure: arcscase1}
\end{subfigure}
\enspace
\begin{subfigure}[b]{0.3\textwidth}
\centering
\def\svgscale{.5}
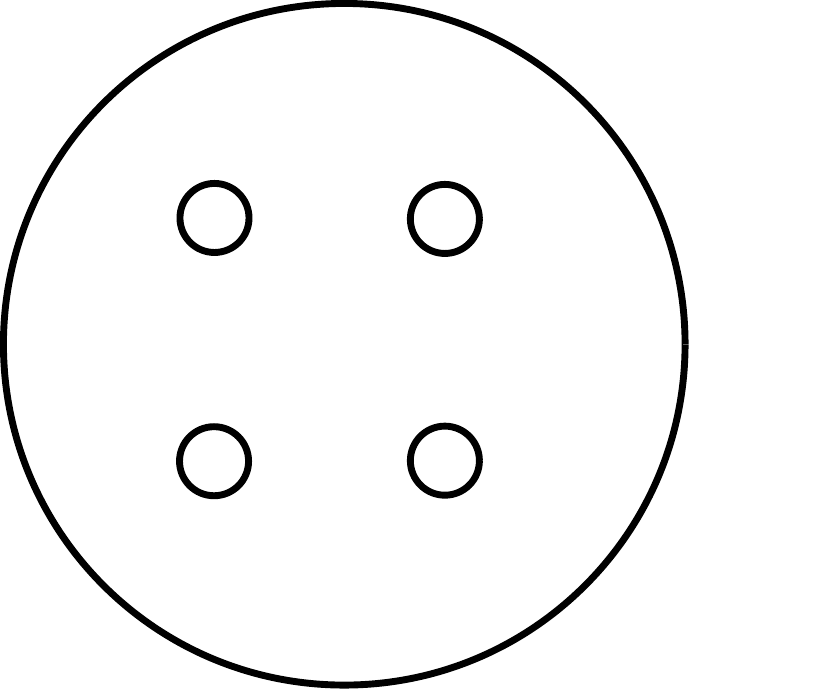
\subcaption{}
\label{figure: arcscase2}
\end{subfigure}
\caption{The subsurface $Y$ in  Claim~\ref{claim:extending DW}, with the surgeries used in the proof.}
\label{figure: arcsurgeries}
\end{figure}

First, assume $Y$ contains the boundary of $S$ and both $a_1$ and $a_2$ separate a 1\hyp{}holed disk containing~$\partial S$ from~$Y$. Since $a_1$ and $a_2$ are disjoint, this implies $a_1$ and $a_2$ are in fact isotopic, with endpoints on a curve $\gamma$ in $\partial_S Y$ (see Figure~\ref{figure: arcscase1}).
Let $D$ be the component of $Y \sminus (a_1 \cup a_2)$ that contains all components of $\partial_S Y \sminus \gamma$.
Since $\partial_S Y$ consists of $g+1\ge4$ curves, there is a curve $\alpha$ on $D$ that separates at least two components of $\partial_S Y$ on each side.  Any such curve is therefore admissible and disjoint from $a_1$ and $a_2$. Moreover, we can choose $\alpha$ to also be disjoint from any curves in~$m' \cap Y$.

Now suppose, without loss of generality, no component of $Y \sminus a_1$ is a 1\hyp{}holed disk containing $\partial S$. There are two standard surgeries of $a_1$ to a non\hyp{}peripheral curve in~$Y$; see Figure~\ref{figure: arcscase2}. Since $\partial_S Y$ has $g+1\geq 4$ components, at least one of these standard surgeries gives an admissible curve of~$Y$. Since $m' \cap Y$ is a disjoint collection of two arcs and some curves, this essential curve intersects $m'$ at most twice.
\end{proof}

Using Claim~\ref{claim:extending DW} on each witness component of $S \sminus m$, we find a vertex $x$ in $P(m)$ with $i(x,m') \leq 8 = 4+2\cdot2$. 
Since the intersection of $x$ with each component of $S \sminus m'$ is a disjoint collection of at most four arcs and possibly some curves, we can repeat the argument of Claim~\ref{claim:extending DW} to show that each witness component of $S \sminus m'$ contains an admissible curve that intersects $x$ at most $6$ times. Thus, there exists a vertex $x'$ in $P(m')$ with $i(x,x') \leq 20 = 8+2\cdot 6$.

Up to the action of $\mcg(S)$, there are only finitely many pairs of vertices of $\ksep(S)$ that intersect at most $20$ times.
Hence, there exists an upper bound $C$, depending only on the genus of $S$, on the distance between such a pair of points. In particular, $d_{\ksep(S)}(x, x') \le C$.
Since $x \in \mathcal{X}(\mu)$ and $x' \in \mathcal{X}(\mu')$, we have that $x$ is contained in $\mathcal{N}_C(\mathcal{X}(\mu)) \cap \mathcal{N}_C(\mathcal{X}(\mu'))$.

To show that $\mathcal{N}_C(\mathcal{X}(\mu)) \cap \mathcal{N}_C(\mathcal{X}(\mu'))$ has infinite diameter, we use the fact that $\Push(S) < \mcg(S)$ contains a pseudo-Anosov element $\phi$ \cite[Theorem~2']{kra}.
Since $\phi$ is a point push, $F(\phi(m)) = F(m)$. Thus, $\phi$ preserves $F^{-1}(\mu)$ and hence preserves~$\mathcal{X}(\mu)$.
Similarly, $\phi$ preserves $\mathcal{X}(\mu')$.
Since the action of $\mcg(S)$ on $\ksep(S)$ is isometric, $\phi$ preserves $\mathcal{N}_C(\mathcal{X}(\mu))$ and $\mathcal{N}_C(\mathcal{X}(\mu'))$, and $\phi^n(x) \in \mathcal{N}_C(\mathcal{X}(\mu)) \cap \mathcal{N}_C(\mathcal{X}(\mu'))$ for all $n \in \mathbb{Z}$.
By Corollary~\ref{corollary:pA_are_undistorted}, the orbit $\langle \phi \rangle \cdot x$ quasi\hyp{}isometrically embeds in $\ksep(S)$ and $\mathcal{N}_C(\mathcal{X}(\mu)) \cap \mathcal{N}_C(\mathcal{X}(\mu'))$ has infinite diameter.
\end{proof}

\begin{lemma}\label{lem:DW0_is_connected}
The graph $\F$ is connected for all $g \geq 3$.
\end{lemma}

\begin{proof}
We use Putman's trick (Lemma \ref{lem:putman_trick}) to establish connectedness. As the generating set for $\mcg(\Sigma)$, select the  Humphries generators and their inverses, that is, the collection of left and right Dehn twists around each of the curves  shown in gray in Figure \ref{fig:humphries}.

\begin{figure}[h!]
    \centering
    \def\svgscale{.5}
    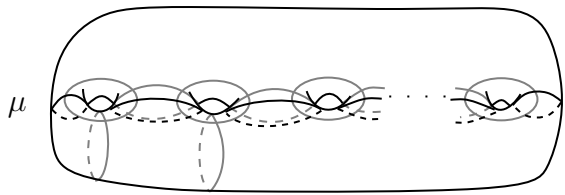
    \caption{Twists generating $\mcg(\Sigma)$ (gray) and the multicurve $\mu$ (black)  }
    \label{fig:humphries}
\end{figure}

Fix a base vertex $\mu$ intersecting the curves of the generating twists as shown in Figure \ref{fig:humphries}.
The action of $\mcg(\Sigma)$ on~$\F$ has only one orbit of vertices, so  the first condition of Lemma~\ref{lem:putman_trick} is satisfied. For the second condition, observe that the image of $\mu$ under every generator intersects $\mu$ at most four times, and hence is connected to $\mu$ by an edge in $\F$.
\end{proof}

We now conclude our proof that $\sep(S)$ is thick of order at most 2.

\begin{proof}[Proof of Theorem \ref{theorem: sep is thick}]
Recall,  it suffices to show $\ksep(S)$ is thick of order at most 2, as the order of thickness is a quasi\hyp{}isometry invariant.

By Lemma~\ref{lem:thick_of_order_1_pieces}, each $\mc{X}(\mu)$ is thick of order~1 for all $\mu \in \FO$. Since there are only finitely many {$\mcg(S)$\hyp{}orbits} of vertices of $\ksep(S)$, there exists a bound $D$ on the diameter of the quotient of $\ksep(S)$ by the action of $\mcg(S)$. 
Using the isometric action of the mapping class group, for any vertex  $x \in \ksep(S)$ there is $\mu \in \FO$ so that $x$ is within distance~$D$ of a vertex of $\mc{X}(\mu)$.

What remains to be shown is that any two of the $\mc{X}(\mu)$'s can be thickly chained together. By Lemma~\ref{lem:DW0_is_connected}, any two vertices $\mu, \mu'\in \F$ are connected by a path $\mu=\mu_0,\mu_1,\dots,\mu_k = \mu'$ in $\F$. By Proposition~\ref{prop:intersection_thick_of_order_2}, there exists $C \geq 0$ independent of $\mu$ and $\mu'$, so that $\mathcal{N}_C(\mathcal{X}(\mu_i)) \cap \mathcal{N}_C(\mathcal{X}(\mu_{i+1}))$ has infinite diameter for all $0\leq i \leq k-1$. Thus, $\ksep(S)$ is thick of order at most 2.
\end{proof}

\bibliography{Bibliography}{}
\bibliographystyle{alpha}

\end{document}